\documentclass[11pt]{article}

\usepackage[a4paper,margin=1in]{geometry}
\usepackage{amsmath,amssymb,amsthm,mathtools}
\usepackage{authblk}
\usepackage{enumitem}
\usepackage{graphicx}
\usepackage{mathrsfs}
\usepackage{microtype}
\usepackage{algorithm}
\usepackage{algpseudocode}
\usepackage{hyperref}

\hypersetup{
  colorlinks=true,
  linkcolor=blue,
  citecolor=blue,
  urlcolor=blue
}

\numberwithin{equation}{section}

\newtheorem{theorem}{Theorem}[section]
\newtheorem{lemma}[theorem]{Lemma}
\newtheorem{proposition}[theorem]{Proposition}
\newtheorem{corollary}[theorem]{Corollary}
\theoremstyle{definition}
\newtheorem{definition}[theorem]{Definition}
\newtheorem{assumption}[theorem]{Assumption}
\newtheorem{problem}[theorem]{Problem}
\theoremstyle{remark}
\newtheorem{remark}[theorem]{Remark}

\newcommand{\E}{\mathbb{E}}
\newcommand{\Prob}{\mathbb{P}}

\title{Single-Chord Augmentation of Weighted Cycles for Algebraic Connectivity and Network Coherence}
\author[1]{Jiarong Deng}
\author[,2]{Liu Chang\thanks{Corresponding author: Liu Chang (liuchang\_@nudt.edu.cn)}}
\author[3,4]{Quanshun Yang}
\affil[1]{School of Mathematics and Statistics, Jishou University,\linebreak Jishou 416000, China}
\affil[2]{College of Sciences, National University of Defense Technology,\linebreak Changsha 410073, China}
\affil[3]{National Key Laboratory of Electromagnetic Energy, \linebreak Naval University of Engineering, Wuhan 430033, China}
\affil[4]{East Lake Laboratory, Wuhan 430033, China}

\date{\today}

\begin{document}

\maketitle

\begin{abstract}
	Ring-like communication graphs appear in UAV formations, cyclic patrols, perimeter monitoring, and other multi-agent tasks in which agents exchange information mainly with neighboring vehicles along a closed route. When measurement and actuation noise are persistent, a useful augmentation should improve both the convergence rate of consensus and the steady-state disagreement level. This paper studies the addition of a single weighted chord to a connected weighted cycle. The central observation is that a chord is not just a generic rank-one edge update: it splits the cycle into two complementary resistance arcs, and this resistance split governs both the algebraic-connectivity gain and the Kirchhoff-index reduction. We first derive exact chord-induced effective-resistance and Kirchhoff-index update formulas, giving a closed-form coherence objective. We then prove that, under bounded conductances and small resistance discrepancy, near-antipodal resistance-balanced chords are near-optimal for algebraic-connectivity improvement; an i.i.d. bounded-conductance model yields the same conclusion with high probability. Finally, because the best convergence-rate chord and the best coherence chord need not coincide, we formulate the design as a finite Pareto problem and introduce RBAPS and AW-RBAPS, two resistance-balanced screening rules that retain only linear or near-linear candidate sets. Numerical experiments show that AW-RBAPS remains effective beyond the formal moderate-heterogeneity regime and approximates the exhaustive Pareto front with mean hypervolume ratio $0.9987$ while evaluating about $10.1\%$ of admissible chords.
\end{abstract}

\textbf{Keywords}: Multi-agent systems; UAV communication networks; weighted cycles; algebraic connectivity; network coherence; Kirchhoff index; Pareto screening.

\section{Introduction}
\label{sec:introduction}
Consensus is a basic coordination layer in multi-agent systems, including sensor networks, robotic teams, UAV formations, and distributed monitoring systems~\cite{BulloCortesMartinez2009,OlfatiSaberFaxMurray2007,RenBeard2005}. In many UAV and mobile-sensor missions the communication graph is naturally close to a ring: aircraft may fly around a perimeter, follow a cyclic patrol route, maintain a circular formation around a target, or relay information along a closed sensing chain. Such cycle-like graphs are sparse, robust to a single local communication failure, and easy to maintain under limited onboard energy and bandwidth. At the same time, a pure ring can have slow information mixing and poor disturbance attenuation. A practical way to improve performance is therefore to add a small number of nonlocal communication links. This paper focuses on the most elementary and analytically transparent case: adding one chord to a weighted cycle.

The control performance of this augmentation has two different aspects. For noisy first-order consensus, the algebraic connectivity determines the exponential decay rate of the mean disagreement, whereas the first-order network coherence is proportional to the Kirchhoff index and measures the steady-state disagreement created by persistent noise~\cite{Bamieh2012,Young2010}. Both quantities are standard in networked control and spectral graph theory~\cite{Fiedler1973,Klein1993,GutmanMohar1996}. They are related but not interchangeable: a chord that maximizes convergence-rate improvement need not maximize coherence improvement.

A tempting approach is to treat chord insertion as a generic edge-addition problem. Existing methods based on Fiedler-vector separation, rank-one perturbation, resistance-distance updates, or greedy edge selection provide useful global tools~\cite{Ghosh2006,Kim2010,Sardar2025,Sydney2013,WanEtAl2008}. However, a weighted cycle contains additional geometry that these graph-generic rules do not explicitly use. Every admissible chord divides the ring into two complementary arcs. The total resistances of these two arcs determine the endpoint effective resistance, shape the low-frequency Laplacian modes, and enter the rank-one update of the Kirchhoff index. Thus, in this graph class, the chord location should be evaluated through complementary-arc resistance balance rather than through endpoint distance or a single spectral vector alone.

The paper is organized around the following three-part logic.
\begin{itemize}[leftmargin=1.5em]
	\item \emph{Problem.} We identify the cycle-specific design variable that complements	graph-agnostic edge-addition viewpoints: the complementary-arc resistance split induced by the chord.
	\item \emph{Theory.} We show that the same complementary-arc resistance geometry controls both performance metrics. It yields an exact coherence formula and a near-antipodal resistance-balance theory for algebraic-connectivity gain.
	\item \emph{Algorithm.} We convert the theory into RBAPS and AW-RBAPS, screening rules that evaluate only linear or near-linear candidate sets while preserving the high-quality Pareto trade-off region in experiments.
\end{itemize}
The analysis is deliberately focused on the communication topology and consensus layer rather than on full UAV flight dynamics. This scope makes the results directly applicable to formation, patrol, and sensing architectures in which a cycle-like communication backbone is designed first and then used by higher-level motion controllers.

The main contributions are threefold.
\begin{enumerate}[leftmargin=1.6em]
	\item We derive exact chord-induced effective-resistance and Kirchhoff-index update formulas for weighted cycles. The resulting expression gives an explicit coherence objective in terms of the two complementary arc resistances and the chord conductance.
	\item We develop a near-antipodal resistance-balance theory for algebraic-connectivity gain. A cycle-adapted secular equation and a two-mode comparison theorem show that, under bounded conductances and small resistance discrepancy, a near-antipodal resistance-balanced chord is near-optimal; an i.i.d. bounded-conductance model gives a high-probability version.
	\item We formulate one-chord augmentation as a finite Pareto problem and propose RBAPS/AW-RBAPS screening. The screened rules exploit resistance geometry to retain a near-linear candidate set and empirically recover almost the same Pareto front as exhaustive search.
\end{enumerate}

\textbf{Organization.}
Section~\ref{sec:prelim} defines the noisy consensus model and the one-chord design problem. Section~\ref{sec:resistance} gives the exact coherence update. Section~\ref{sec:algconn} develops the resistance-balanced algebraic-connectivity theory, with technical lemmas and proofs placed in the appendix. Section~\ref{sec:pareto} formulates the finite Pareto front, and Section~\ref{sec:experiments} evaluates RBAPS and AW-RBAPS. Section~\ref{sec:conclusion} concludes the paper.

\section{Preliminaries and Problem Statement}
\label{sec:prelim}

\subsection{Weighted graphs, Laplacian, and algebraic connectivity}

Let $\mathcal{G}^{w}=(V,E,A)$ be a connected undirected weighted graph on
$n$ vertices $V=\{v_{0},\dots,v_{n-1}\}$.
Its symmetric adjacency matrix $A=[a_{ij}]$ satisfies $a_{ij}\ge 0$, with
$a_{ij}>0$ if and only if $\{v_i,v_j\}\in E$.
The weighted degree of vertex $v_i$ is $d_i\coloneqq \sum_{j=0}^{n-1} a_{ij}$ and the associated degree matrix is $D\coloneqq \operatorname{diag}(d_0,\dots,d_{n-1})$.
The weighted Laplacian of $\mathcal{G}^{w}$ is $L\coloneqq D-A$.

Since $\mathcal{G}^{w}$ is connected, $L$ is symmetric positive
semidefinite, with a simple zero eigenvalue associated with the all-ones
vector $\boldsymbol{1}=(1,\dots,1)^\top$.
We order the eigenvalues of $L$ as $0=\lambda_0<\lambda_1\le \lambda_2\le \cdots \le \lambda_{n-1}$,
and denote by $\{(\lambda_k,\boldsymbol{u}_k)\}_{k=0}^{n-1}$ an orthonormal
eigenbasis, with $\boldsymbol{u}_0=\boldsymbol{1}/\sqrt{n}$.
The second smallest eigenvalue $\lambda_1$ is the
\emph{algebraic connectivity}~\cite{Fiedler1973}, and any corresponding
eigenvector $\boldsymbol{u}_1\perp \boldsymbol{1}$ is called a
\emph{Fiedler vector}.
By the Courant--Fischer theorem,
\begin{equation}\label{eq:rr_lambda1}
	\lambda_1=
	\min_{\substack{\boldsymbol{x}\perp \boldsymbol{1}\\ \|\boldsymbol{x}\|_2=1}}
	\boldsymbol{x}^\top L \boldsymbol{x}.
\end{equation}

The disagreement subspace is $\boldsymbol{1}^\perp
\coloneqq
\{\boldsymbol{x}\in\mathbb{R}^n:\boldsymbol{1}^\top \boldsymbol{x}=0\}$.
For each vertex $v_i$, we write $\boldsymbol{e}_i\in\mathbb{R}^n$ for the
$i$th canonical basis vector.

\subsection{Noisy first-order consensus and network coherence}

Following~\cite{Bamieh2012}, consider $n$ agents interacting over the
weighted graph $\mathcal{G}^{w}$.
Let $\xi_i(t)\in\mathbb{R}$ denote the state of agent $i$, and define $\boldsymbol{\xi}(t)\coloneqq
(\xi_0(t),\dots,\xi_{n-1}(t))^\top$. The noisy first-order consensus dynamics are
\begin{equation}\label{eq:consensus_sde}
	\mathrm{d}\boldsymbol{\xi}(t)
	=
	-L\,\boldsymbol{\xi}(t)\,\mathrm{d}t
	+
	\sigma\,\mathrm{d}\mathbf{W}(t),
\end{equation}
where $\sigma>0$ is the noise intensity and
$\mathbf{W}(t)=(W_0(t),\dots,W_{n-1}(t))^\top$ is a standard
$n$-dimensional Brownian motion.

To isolate disagreement dynamics, introduce the orthogonal projector: $P\coloneqq I-\frac{1}{n}\boldsymbol{1}\boldsymbol{1}^\top$ and $\widetilde{\boldsymbol{\xi}}(t)\coloneqq P\boldsymbol{\xi}(t).$
Since $PL=LP=L$, the projected process satisfies
\begin{equation}\label{eq:disag_dyn}
	\mathrm{d}\widetilde{\boldsymbol{\xi}}(t)
	=
	-L\,\widetilde{\boldsymbol{\xi}}(t)\,\mathrm{d}t
	+
	\sigma P\,\mathrm{d}\mathbf{W}(t),
	\quad
	\widetilde{\boldsymbol{\xi}}(0)\in \boldsymbol{1}^\perp.
\end{equation}

The next result summarizes standard facts on the mean decay and steady-state
covariance of the disagreement process~\cite{Bamieh2012}.

\begin{proposition}[Mean decay and steady-state covariance]
	\label{prop:consensus_covariance}
	For the disagreement dynamics \eqref{eq:disag_dyn}, the following hold.
	\begin{enumerate}
		\item[\textnormal{(i)}]
		\[
		\mathbb{E}[\widetilde{\boldsymbol{\xi}}(t)]
		=
		e^{-Lt}\widetilde{\boldsymbol{\xi}}(0),
		\quad
		\|\mathbb{E}[\widetilde{\boldsymbol{\xi}}(t)]\|_2
		\le
		e^{-\lambda_1 t}\|\widetilde{\boldsymbol{\xi}}(0)\|_2.
		\]
		\item[\textnormal{(ii)}]
		The disagreement second-moment matrix $\Sigma(t)\coloneqq
		\mathbb{E}\!\left[
		\widetilde{\boldsymbol{\xi}}(t)
		\widetilde{\boldsymbol{\xi}}(t)^\top
		\right]$
		satisfies $\frac{\mathrm{d}}{\mathrm{d}t}\Sigma(t)
		=
		-L\Sigma(t)-\Sigma(t)L+\sigma^2P$.
		When the initial disagreement mean is zero, $\Sigma(t)$ is the covariance
		matrix. In all cases, its unique steady-state covariance on
		$\boldsymbol{1}^\perp$ is
		\begin{equation}\label{eq:sigma_infty}
			\Sigma_\infty
			=
			\frac{\sigma^2}{2}L^\dagger,
			\qquad
			L^\dagger
			=
			\sum_{k=1}^{n-1}
			\frac{1}{\lambda_k}\boldsymbol{u}_k\boldsymbol{u}_k^\top,
		\end{equation}
		where $L^\dagger$ is the Moore--Penrose pseudoinverse of $L$.
	\end{enumerate}
\end{proposition}

Proposition~\ref{prop:consensus_covariance} shows that $\lambda_1$
determines the exponential decay rate of the mean disagreement, whereas
$L^\dagger$ governs the steady-state fluctuations induced by the noise.
This motivates the standard notion of network coherence~\cite{Bamieh2012}.

\begin{definition}[First-order network coherence \cite{Bamieh2012}]
	The \emph{first-order network coherence} is the average steady-state
	variance of the deviation from the instantaneous network average $\bar{\xi}(t)\coloneqq \frac{1}{n}\sum_{i=0}^{n-1}\xi_i(t)$:
	\begin{equation}\label{eq:coherence_def}
		H
		\coloneqq
		\lim_{t\to\infty}
		\frac{1}{n}
		\sum_{i=0}^{n-1}
		\mathbb{E}\bigl[(\xi_i(t)-\bar{\xi}(t))^2\bigr].
	\end{equation}
\end{definition}

Using \eqref{eq:sigma_infty}, one obtains
\begin{equation}\label{eq:coherence_trace}
	H
	=
	\frac{1}{n}\operatorname{tr}(\Sigma_\infty)
	=
	\frac{\sigma^2}{2n}\operatorname{tr}(L^\dagger)
	=
	\frac{\sigma^2}{2n}\sum_{k=1}^{n-1}\frac{1}{\lambda_k}.
\end{equation}
Hence smaller $H$ means that the agent states remain more tightly clustered
around the network average despite the stochastic disturbances.

\subsection{Effective resistance, Kirchhoff index, and coherence}

Two classical graph quantities that naturally enter the analysis are the
effective resistance and the Kirchhoff index~\cite{Klein1993}.

\begin{definition}[Effective resistance and Kirchhoff index]
	\label{def:resistance}
	The \emph{effective resistance} between vertices $v_i$ and $v_j$ is
	\begin{equation}\label{eq:resistance_def}
		R_{ij}
		\coloneqq
		(\boldsymbol{e}_i-\boldsymbol{e}_j)^\top
		L^\dagger
		(\boldsymbol{e}_i-\boldsymbol{e}_j).
	\end{equation}
	The \emph{Kirchhoff index} of $\mathcal{G}^{w}$ is $K_f
	\coloneqq
	\sum_{0\le i<j\le n-1} R_{ij}$.
\end{definition}

A standard spectral identity~\cite{Bapat2010} yields
\begin{equation}\label{eq:Kf_trace}
	K_f
	=
	n\,\operatorname{tr}(L^\dagger)
	=
	n\sum_{k=1}^{n-1}\frac{1}{\lambda_k}.
\end{equation}
Combining \eqref{eq:coherence_trace} and \eqref{eq:Kf_trace}, we obtain
\[
H=\frac{\sigma^2}{2n^2}K_f.
\]
Therefore, minimizing the network coherence is equivalent to minimizing the
Kirchhoff index.

For the noisy consensus model, pairwise steady-state variances are also
linked to effective resistances through
\begin{equation}\label{eq:var_resistance}
	\mathbb{E}\bigl[(\xi_i-\xi_j)^2\bigr]_{\mathrm{ss}}
	=
	\frac{\sigma^2}{2}R_{ij}.
\end{equation}
Accordingly, effective resistance serves as a key intermediate quantity in
our subsequent analysis of coherence and chord augmentation.

\subsection{Problem statement: one-chord augmentation on a weighted cycle}

We now specialize to a weighted cycle
$\mathcal{C}_n$ on vertices $v_0,\dots,v_{n-1}$, with strictly positive edge
conductances $c_{i,i+1}>0$ for $i=0,\dots,n-1$, where indices are taken
modulo $n$.
Its Laplacian is denoted by $L$.

We augment the cycle by adding a single extra weighted edge, referred to as
a \emph{chord}, between two distinct non-adjacent vertices $v_p$ and $v_q$.
If the chord conductance is $w\ge 0$, then the augmented Laplacian is the
rank-one update
\begin{equation}\label{eq:chord_update}
	L_{p,q}(w)
	\coloneqq
	L+w\,\boldsymbol{b}_{pq}\boldsymbol{b}_{pq}^\top,
	\qquad
	\boldsymbol{b}_{pq}\coloneqq \boldsymbol{e}_p-\boldsymbol{e}_q.
\end{equation}
When no confusion arises, we simply write $\boldsymbol{b}$ for
$\boldsymbol{b}_{pq}$.

The design question studied in this paper is to determine how the chord
location, and when allowed the chord conductance, affects two fundamental
but generally distinct performance criteria.

Let
\begin{equation}\label{eq:feasible_chord_set}
	\mathcal{E}_{\mathrm{ch}}
	\coloneqq
	\bigl\{\{p,q\}:0\le p<q\le n-1,\ \{v_p,v_q\}\notin E(\mathcal C_n)\bigr\}
\end{equation}
denote the set of admissible unordered nonadjacent vertex pairs.

For algebraic connectivity, we quantify the benefit of adding the chord
$\{p,q\}\in\mathcal{E}_{\mathrm{ch}}$ by
\begin{equation}\label{eq:delta_pq_def}
	\Delta_{p,q}(w)
	\coloneqq
	\lambda_1\bigl(L_{p,q}(w)\bigr)-\lambda_1(L)\ge 0.
\end{equation}
By eigenvalue interlacing, this gain is bounded above by
\begin{equation}\label{eq:interlacing_ceiling_problem}
	0\le \Delta_{p,q}(w)\le
	\gamma,
	\qquad
	\gamma\coloneqq \lambda_2(L)-\lambda_1(L),
\end{equation}
for every admissible pair $\{p,q\}\in\mathcal{E}_{\mathrm{ch}}$ and every $w\ge0$.

For network coherence, equivalently for the Kirchhoff index, we define the
improvement
\begin{equation}\label{eq:I_pq_def}
	\mathcal{I}_{p,q}(w)
	\coloneqq
	K_f(L)-K_f\bigl(L_{p,q}(w)\bigr)\ge 0.
\end{equation}
Since $H=\sigma^2K_f/(2n^2)$, maximizing $\mathcal I_{p,q}(w)$ is
equivalent to maximizing the reduction in first-order network coherence.

\begin{problem}\emph{(One-chord augmentation on a weighted cycle).}
	Given a connected weighted cycle $\mathcal C_n$ and a chord conductance
	budget, characterize and compare the admissible chords
	$\{p,q\}\in\mathcal{E}_{\mathrm{ch}}$ that
	\begin{enumerate}
		\item maximize the algebraic-connectivity gain
		$\Delta_{p,q}(w)$, and
		\item maximize the Kirchhoff-index reduction
		$\mathcal I_{p,q}(w)$, equivalently minimize the resulting
		first-order network coherence.
	\end{enumerate}
\end{problem}

The two objectives are related but generally distinct. The remainder of the paper therefore separates the closed-form coherence calculation from the spectral convergence-rate analysis, and then recombines them in a Pareto formulation. This separation is important for applications: a UAV communication chord that is best for fast transient agreement may differ from the chord that best attenuates persistent sensing or actuation noise.

\section{Effective Resistance and Kirchhoff Index}\label{sec:resistance}

\subsection{Notation for the weighted cycle}

Set
\begin{equation*}
	r_s:=\frac1{c_{s,s+1}},\qquad s=0,\dots,n-1,
	\qquad
	T:=\sum_{s=0}^{n-1}r_s.
\end{equation*}
For $a,b\in\{0,\dots,n-1\}$, define the directed arc resistance
\begin{equation}\label{eq:directed_arc_resistance}
	d(a,b):=
	\begin{cases}
		\displaystyle\sum_{s=a}^{b-1}r_s, & a<b,\\[0.7em]
		0, & a=b,\\[0.7em]
		\displaystyle\sum_{s=a}^{n-1}r_s+\sum_{s=0}^{b-1}r_s, & a>b.
	\end{cases}
\end{equation}
Then $d(a,b)+d(b,a)=T$.
For $0\le p<q\le n-1$, write
\begin{equation}\label{eq:Apq_Bpq_def}
	A_{pq}:=d(p,q),
	\qquad
	B_{pq}:=d(q,p),
	\qquad
	A_{pq}+B_{pq}=T.
\end{equation}

\subsection{Exact update formulas}

\begin{lemma}[Resistance on the weighted cycle]\label{lem:Rab_cycle}
	For any two vertices $v_a,v_b$,
	\begin{equation}\label{eq:Rab_cycle}
		R_{ab}=\frac{d(a,b)d(b,a)}{T}.
	\end{equation}
	In particular, $R_{pq}=A_{pq}B_{pq}/T$.
\end{lemma}

\begin{lemma}\label{lem:res_kir_update}
	\emph{(Exact update of effective resistance and Kirchhoff index).}
	Let $L_{p,q}(w)=L+w\boldsymbol{b}\boldsymbol{b}^\top$ with $\boldsymbol{b}=\boldsymbol{e}_p-\boldsymbol{e}_q$ and $w\ge0$.
	Then:
	\begin{enumerate}
		\item[\textnormal{(i)}]
		\begin{equation}\label{eq:Rpq_update_cycle}
			R_{pq}(w)=\frac{R_{pq}}{1+wR_{pq}}
			=\frac{A_{pq}B_{pq}}{T+wA_{pq}B_{pq}}.
		\end{equation}
		\item[\textnormal{(ii)}] For any unordered pair $\{u,v\}$,
		\begin{equation}\label{eq:Ruv_update_cycle_compact}
			R_{uv}(w)=R_{uv}-\frac{w}{4(1+wR_{pq})}
			\bigl(R_{uq}+R_{vp}-R_{up}-R_{vq}\bigr)^2.
		\end{equation}
		\item[\textnormal{(iii)}]
		\begin{equation}\label{eq:Kf_cycle_sum}
			K_f(L)=\frac1T\sum_{0\le u<v\le n-1}d(u,v)(T-d(u,v)),
		\end{equation}
		and
		\begin{equation}\label{eq:Kf_update_cycle_compact}
			K_f(L_{p,q}(w))
			=K_f(L)-\mathcal I_{p,q}(w),
		\end{equation}
		where the exact improvement is
		\begin{equation}\label{eq:Kf_improvement_full}
			\mathcal I_{p,q}(w):=
			\frac{w\sum_{0\le u<v\le n-1}
				\bigl(R_{uq}+R_{vp}-R_{up}-R_{vq}\bigr)^2}{4(1+wR_{pq})}.
		\end{equation}
	\end{enumerate}
\end{lemma}

\subsection{Monotonicity and design implications}

\begin{theorem}\label{thm:R_Kf_quant}
	\emph{(Quantitative dependence of $R_{pq}$ and $K_f$ on $(p,q,w)$).}
	For every fixed pair $(p,q)$ and every $w\ge0$:
	\begin{enumerate}
		\item[\textnormal{(i)}] $\frac{\mathrm{d}}{\mathrm{d}w}R_{pq}(w)=-\frac{R_{pq}^2}{(1+wR_{pq})^2}<0$.
		Thus $R_{pq}(w)$ is strictly decreasing and convex in $w$.
		\item[\textnormal{(ii)}] The endpoint reduction $\Delta R_{pq}(w):=R_{pq}-R_{pq}(w)=\frac{wR_{pq}^2}{1+wR_{pq}}$ is strictly increasing in $w$.
		\item[\textnormal{(iii)}] $K_f(L_{p,q}(w))$ is strictly decreasing in $w$, and the improvement $\mathcal I_{p,q}(w)$ is increasing and concave in $w$.
	\end{enumerate}
\end{theorem}

\begin{corollary}\label{cor:selection_RK}
	\emph{(Optimal chord for resistance and coherence).}
	Fix a conductance budget $0\le w\le \hat{w}$.
	\begin{enumerate}
		\item[\textnormal{(i)}]  The endpoint resistance reduction $\Delta R_{pq}(w)=R_{pq}-R_{pq}(w)$ is maximized by taking $w=\hat w$ and choosing a chord that maximizes $R_{pq}=A_{pq}B_{pq}/T$.
		
		\item[\textnormal{(ii)}] The Kirchhoff index, and hence the first-order network coherence, is minimized by taking $w=\hat{w}$ and choosing $(p,q)$ that maximizes the full improvement
		\begin{equation*}
			\mathcal {I}_{p,q}(\hat{w})
			=
			\frac{\hat{w}\sum_{0\le u<v\le n-1}\bigl(R_{uq}+R_{vp}-R_{up}-R_{vq}\bigr)^2}{4(1+\hat{w} R_{pq})}.
		\end{equation*}
		In general, it is not sufficient to maximize only the square-sum term, because the denominator $1+\hat{w} R_{pq}$ also depends on the candidate chord.
	\end{enumerate}
\end{corollary}

\begin{remark}\label{remark:budget}\emph{(Budget saturation and native-scale chord weight).}
	For a fixed admissible chord $e=(p,q)$, increasing $w$ cannot
	decrease the algebraic connectivity because
	$L+w \boldsymbol{b}_e \boldsymbol{b}_e^\top$ is monotone in the Loewner order. It also strictly	increases the Kirchhoff-index improvement by Theorem \ref{thm:R_Kf_quant}.
	Thus, under a budget $0\le w\le \hat w$, we always use the full
	budget $w=\hat w$. For the near-optimality theory, we assume that this full-budget chord is on the native resistance scale:
	$\frac{1}{\hat w}\le C_w\bar r$, $\bar r:=\frac{1}{n}\sum_{i=0}^{n-1}r_i .$
	That is, the added chord resistance is at most a constant multiple of
	the average resistance of a native cycle edge. This lower-scale
	condition excludes asymptotically negligible chords and supplies the
	fixed lower bound on $w$ used in the spectral analysis.
\end{remark}

\begin{remark}\emph{(Uniform cycle with one chord).}
	For a uniform cycle where every edge has conductance $c>0$, 
	the Kirchhoff index improvement due to a chord of conductance $w$
	connecting vertices at graph distance $d$ (with $1\le d\le \lfloor n/2\rfloor$)
	is monotone in both $d$ and $w$.
	Consequently, under a conductance budget $0\le w\le \hat{w}$,
	the chord that minimizes the Kirchhoff index
	(and thus network coherence) is the antipodal one 
	($d=\lfloor n/2\rfloor$) with the maximal allowed weight $w=\hat{w}$.
	The same choice also gives the largest reduction of the endpoint
	effective resistance between the two vertices.
\end{remark}

\section{Algebraic Connectivity}\label{sec:algconn}

This section studies how adding a single chord changes the algebraic
connectivity of a weighted cycle. We first collect general facts for the
rank-one Laplacian update
\[
L_{p,q}(w)=L+w\boldsymbol{b}_{pq}\boldsymbol{b}_{pq}^{\top},
\qquad
\boldsymbol{b}_{pq}:=\boldsymbol{e}_p-\boldsymbol{e}_q,
\]
then specialize them to weighted cycles. The argument has three steps:
i) abstract rank-one perturbation analysis; ii) two-terminal reduction of
the cycle; iii) verification of the abstract hypotheses under deterministic
discrepancy control. Throughout this section, $(p,q)$ denotes a pair of
distinct vertices, and $w\ge 0$.

\subsection{Rank-one perturbation facts}

We begin with standard spectral facts for the update
$L+w\boldsymbol{b}\boldsymbol{b}^{\top}$. These results hold for any
connected weighted graph. Recall that $L$ has orthonormal eigenpairs
$(\lambda_k,\boldsymbol{u}_k)$, $k=0,\dots,n-1$, with
$\boldsymbol{u}_0=\boldsymbol{1}/\sqrt{n}$.

\begin{lemma}
	\label{lem:lambda_update}
	\emph{(Secular equation and interlacing \cite{Cve-Doob-Sachs,HornJohnson2012,So1999}).}
	Let $\boldsymbol{b}=\boldsymbol{e}_p-\boldsymbol{e}_q$ and define
	\begin{equation*}
		\beta_k(p,q)\coloneqq \boldsymbol{u}_k^\top\boldsymbol{b}
		= u_{k,p}-u_{k,q},
		\qquad k=1,\dots,n-1.
	\end{equation*}
	An eigenvalue $\mu$ of $L+w\boldsymbol{b}\boldsymbol{b}^\top$ either
	coincides with an eigenvalue of $L$ having an eigenvector orthogonal to
	$\boldsymbol{b}$, or satisfies
	\begin{equation}\label{eq:secular_resolvent}
		1+w\boldsymbol{b}^\top(L-\mu I)^{-1}\boldsymbol{b}=0.
	\end{equation}
	Equivalently, away from the poles of the resolvent,
	\begin{equation}\label{eq:secular_spectral}
		1+w\sum_{k=1}^{n-1}\frac{\beta_k(p,q)^2}{\lambda_k-\mu}=0.
	\end{equation}
	Moreover, the eigenvalues interlace; in particular,
	\begin{equation}\label{eq:interlacing_inequalities}
		0=\lambda_0(L+w\boldsymbol{b}\boldsymbol{b}^\top)<\lambda_1(L)
		\le \lambda_1(L+w\boldsymbol{b}\boldsymbol{b}^\top)
		\le \lambda_2(L).
	\end{equation}
\end{lemma}

\begin{lemma}[Interlacing ceiling for the gain]\label{lem:gain_ceiling}
	For every pair $(p,q)$ and every $w\ge 0$,
	\begin{equation*}
		\begin{split}
			0\le \Delta_{p,q}(w)&:=
			\lambda_1(L_{p,q}(w))-\lambda_1(L)\le \gamma,\\
			\gamma&:=\lambda_2(L)-\lambda_1(L).
		\end{split}
	\end{equation*}
	Moreover, for any unit Fiedler vector $\boldsymbol{u}_1$,
	\[
	\Delta_{p,q}(w)\le w(u_{1,p}-u_{1,q})^2.
	\]
\end{lemma}

\begin{theorem}\label{thm:lambda_quant}
	\emph{(Variational characterization of	the updated algebraic connectivity).}
	Let $\lambda_1(w):=\lambda_1(L_{p,q}(w))$. Then
	\begin{equation}\label{eq:var_lam_w}
		\lambda_1(w)=
		\min_{\substack{\boldsymbol{x}\perp\boldsymbol{1}\\ \|\boldsymbol{x}\|_2=1}}
		\left(\boldsymbol{x}^\top L\boldsymbol{x}+w(x_p-x_q)^2\right).
	\end{equation}
	The map $w\mapsto\lambda_1(w)$ is nondecreasing and concave. If
	$\lambda_1(w)$ is simple with unit eigenvector $\boldsymbol{v}(w)$,
	then
	\[
	\frac{\mathrm{d}}{\mathrm{d}w}\lambda_1(w)=(v_p(w)-v_q(w))^2.
	\]
	Finally,
	\[
	\lim_{w\to\infty}\lambda_1(w)=
	\min_{\substack{\boldsymbol{x}\perp\boldsymbol{1},\ x_p=x_q\\ \|\boldsymbol{x}\|_2=1}}
	\boldsymbol{x}^\top L\boldsymbol{x}.
	\]
\end{theorem}

\subsection{A two-mode comparison theorem}

We next isolate an abstract mechanism for achieving near-maximal gain in
algebraic connectivity. By Lemma~\ref{lem:gain_ceiling}, the largest
possible gain is the interlacing ceiling $\gamma=\lambda_2-\lambda_1$.
The result below shows that this ceiling is nearly attained whenever the
chord induces a large jump in the first mode, a small jump in the second
mode, and a controlled contribution from higher modes.

If $\lambda_1=\lambda_2$, then by interlacing $\Delta_{p,q}(w)\le \lambda_2-\lambda_1=0$, hence all admissible chords have zero algebraic-connectivity gain and the near-optimality statement is trivial. We henceforth assume $\lambda_1<\lambda_2$.

Recall that
$\beta_k(p,q)\coloneqq \boldsymbol{u}_k^\top\boldsymbol{b}=u_{k,p}-u_{k,q}$,
$k=1,\dots,n-1$. For a fixed pair $(p,q)$, define
\begin{equation}\label{eq:T3_tail}
	T_{3+}(p,q):=\sum_{k=3}^{n-1}\frac{\beta_k(p,q)^2}{\lambda_k}.
\end{equation}

From Eqs.~\eqref{eq:sigma_infty} and \eqref{eq:resistance_def}, for any
connected weighted graph and any pair $(p,q)$,
\begin{equation*}
	R_{pq}
	=
	\sum_{k=1}^{n-1}\frac{\beta_k(p,q)^2}{\lambda_k}
	=
	\frac{\beta_1(p,q)^2}{\lambda_1}
	+\frac{\beta_2(p,q)^2}{\lambda_2}
	+T_{3+}(p,q).
\end{equation*}
The comparison theorem below only requires the chord to be non-negligible in strength, namely $w\ge w_0$. For weighted cycles, we express this condition in resistance units: the chord resistance $r_c:=1/w$ should be no larger than a fixed multiple of the mean native edge resistance. This native-scale condition is stated precisely in Assumption~\ref{ass:deterministic}; any separate upper conductance budget is a design constraint and is not part of the spectral comparison argument.

\begin{lemma}[Higher-mode tail and effective resistance]\label{lem:two_mode_resistance}
	If $\lambda_2\le \rho_0\lambda_3$ for some $\rho_0\in(0,1)$, then for
	every $\mu\in[\lambda_1,\lambda_2)$,
	\begin{equation}\label{eq:tail_resolvent_bound}
		\sum_{k=3}^{n-1}\frac{\beta_k(p,q)^2}{\lambda_k-\mu}
		\le \frac{1}{1-\rho_0}T_{3+}(p,q).
	\end{equation}
\end{lemma}

\begin{theorem}\label{thm:comparison_two_mode}
	\emph{(Two-mode secular comparison).}
	Fix a connected weighted graph and a pair $(p,q)$. Let
	\begin{equation*}
		\gamma:=\lambda_2-\lambda_1,
		\quad
		\beta_k:=\beta_k(p,q),
		\quad
		T_{3+}:=T_{3+}(p,q).
	\end{equation*}
	Assume that $\lambda_1<\lambda_2$ and that, for some constants
	\begin{equation*}
		\begin{split}
			&\rho_0\in(0,1),\quad a_0>0,\quad A_0>0,\\
			&L_0>0,\quad w_0>0,\quad \theta_0\in(0,1),
		\end{split}
	\end{equation*}
	one has
	\begin{equation}\label{eq:comparison_hyp1}
		\begin{split}
			\lambda_2\le \rho_0\lambda_3,
			\quad
			&\beta_1^2\ge \frac{a_0}{n},
			\quad
			T_{3+}\le A_0 n,\\
			&\lambda_2\le \frac{L_0}{n^2},
			\quad
			w\ge w_0,
		\end{split}
	\end{equation}
	and
	\begin{equation}\label{eq:comparison_dominance}
		\gamma\left(\frac1w+\frac{T_{3+}}{1-\rho_0}\right)
		\le \theta_0\beta_1^2.
	\end{equation}
	If, for some $\varepsilon\ge 0$,
	\begin{equation}\label{eq:comparison_small_beta2}
		\beta_2^2\le \varepsilon\beta_1^2,
	\end{equation}
	then
	\begin{equation}\label{eq:ceiling_deficit_gamma}
		0\le \gamma-\Delta_{p,q}(w)
		\le \frac{\gamma}{1-\theta_0}\,\varepsilon.
	\end{equation}
	Consequently,
	\begin{equation}\label{eq:ceiling_deficit_wn}
		0\le \gamma-\Delta_{p,q}(w)
		\le C\frac{w}{n}\varepsilon,
	\end{equation}
	where $C$ depends only on $L_0$, $w_0$, and $\theta_0$.
\end{theorem}

Lemma~\ref{lem:gain_ceiling} implies
$\Delta_{u,v}(w)\le \gamma$ for every admissible chord
$\{u,v\}\in\mathcal{E}_{\mathrm{ch}}$, and therefore we have the following corollary.

\begin{corollary}[From ceiling deficit to near-optimality]\label{cor:ceiling_to_opt}
	For every admissible chord $\{p,q\}\in\mathcal{E}_{\mathrm{ch}}$ and every $w\ge 0$,
	\begin{equation}\label{eq:near_opt_from_ceiling}
		0\le
		\max_{\{u,v\}\in\mathcal{E}_{\mathrm{ch}}}\Delta_{u,v}(w)-\Delta_{p,q}(w)
		\le
		\gamma-\Delta_{p,q}(w).
	\end{equation}
\end{corollary}

\subsection{Deterministic discrepancy theory}\label{subsec:deterministic_discrepancy}

We now verify the abstract hypotheses of
Theorem~\ref{thm:comparison_two_mode} for weighted cycles whose edge
conductances are uniformly bounded. The key device is the resistance
arclength parametrization of the cycle. Under small discrepancy, the first
two discrete eigenvectors are close to the first two Fourier modes on the
circle, which in turn identifies near-antipodal chords as nearly optimal.

Throughout this subsection, write $c_i:=c_{i,i+1}$ and $r_i:=1/c_i$.

\begin{assumption}[Bounded conductances and native-scale chord]\label{ass:deterministic}
	The edge conductances of the weighted cycle satisfy
	\begin{equation}\label{eq:det_kappa_bounds}
		1\le c_i\le \kappa,
		\qquad i=0,\dots,n-1,
	\end{equation}
	so that the native edge resistances obey $r_i=1/c_i\in[\kappa^{-1},1]$.
	The available chord budget is taken on the native scale:
	\begin{equation}\label{eq:w_native_scale}
		\frac{1}{\hat w}\le C_w\bar r ,
	\end{equation}
	where $\bar r:=n^{-1}\sum_i r_i$ and $C_w>0$ is fixed. Since the budget is saturated, the chord used in the analysis has conductance $w=\hat w$.
\end{assumption}

All constants below may depend on $\kappa$ and $C_w$. Set
\[
S\coloneqq \sum_{i=0}^{n-1}r_i,
\qquad
\bar r\coloneqq \frac{S}{n},
\qquad
s_0\coloneqq 0,
\qquad
s_i\coloneqq \sum_{k=0}^{i-1}r_k.
\]
Let $r_{\max}\coloneqq \max_i r_i$, and define the resistance distance
along the cycle by
\[
d_R(p,q)\coloneqq \min\{|s_p-s_q|,\ S-|s_p-s_q|\}.
\]
For $1\le \ell\le n$, define the $\ell$-step cumulative resistance
\[
A_{p,\ell}\coloneqq \sum_{j=0}^{\ell-1}r_{p+j},
\]
with indices modulo $n$. Introduce the discrepancy measures
\begin{equation}\label{eq:def_D}
	D\coloneqq
	\max_{0\le p\le n-1}\max_{1\le \ell\le n}
	\left|A_{p,\ell}-\frac{\ell}{n}S\right|,
	\qquad
	\Delta\coloneqq \frac{D}{S},
	\qquad
	\eta\coloneqq \frac{r_{\max}}{S},
	\qquad
	\delta_n\coloneqq \Delta+\eta.
\end{equation}
Writing $y_i\coloneqq s_i/S$, the definition of $D$ gives the node
discrepancy bound
\begin{equation}\label{eq:node_discrepancy}
	\max_{0\le i\le n}\left|y_i-\frac{i}{n}\right|\le \Delta.
\end{equation}

The technical verification of the small-discrepancy regime is lengthy but standard in spirit: resistance arclength converts the weighted cycle into a nearly uniform mesh on a circle; the first two eigenvectors are then close to the first Fourier sine and cosine modes; a near-antipodal resistance-balanced chord has a large jump in the first mode and a small jump in the second. The detailed quadrature, eigenvector-localization, and verification lemmas are collected in Appendix~\ref{app:auxiliary-results}. The main consequence is the following theorem.

\begin{theorem}[Near-optimal chord selection under deterministic discrepancy control]\label{thm:main_optimal_disc}
	There exist constants $\delta_0=\delta_0(\kappa,C_w)>0$,
	$n_0=n_0(\kappa,C_w)\in\mathbb{N}$, and $C_{\kappa,C_w}>0$ such that
	the following holds. Assume \eqref{eq:det_kappa_bounds},
	\eqref{eq:w_native_scale}, $\delta_n\le \delta_0$, and $n\ge n_0$.
	Let
	\[
	r_{\max}\le \zeta\le S/8,
	\qquad
	\mathcal A_\zeta:=
	\left\{
	(p,q):\left|d_R(p,q)-\frac S2\right|\le \zeta
	\right\},
	\]
	and let $(\widehat p,\widehat q)\in\mathcal A_\zeta$ maximize
	$|\beta_1(p,q)|$ over $\mathcal A_\zeta$. Then
	\begin{equation}\label{eq:main_optimal_disc_gamma}
		0\le
		\max_{\{p,q\}\in\mathcal{E}_{\mathrm{ch}}}\Delta_{p,q}(w)-\Delta_{\widehat p,\widehat q}(w)
		\le
		C_{\kappa,C_w}(\lambda_2-\lambda_1)
		\left[
		\left(\frac{\zeta}{S}\right)^2+\sqrt{\delta_n}
		\right].
	\end{equation}
	In particular,
	\begin{equation}\label{eq:main_optimal_disc_additive}
		0\le
		\max_{\{p,q\}\in\mathcal{E}_{\mathrm{ch}}}\Delta_{p,q}(w)-\Delta_{\widehat p,\widehat q}(w)
		\le
		C_{\kappa,C_w}\frac{w}{n}
		\left[
		\left(\frac{\zeta}{S}\right)^2+\sqrt{\delta_n}
		\right].
	\end{equation}
	For the native choice $\zeta=r_{\max}$, which is admissible for all
	$n\ge 8\kappa$,
	\begin{equation}\label{eq:main_optimal_disc_additive_native}
		0\le
		\max_{\{p,q\}\in\mathcal{E}_{\mathrm{ch}}}\Delta_{p,q}(w)-\Delta_{\widehat p,\widehat q}(w)
		\le
		C_{\kappa,C_w}\frac{w}{n}\sqrt{\delta_n}.
	\end{equation}
\end{theorem}

\begin{corollary}[Ratio-form near-optimality]\label{cor:main_optimal_disc_ratio}
	Assume the hypotheses of Theorem~\ref{thm:main_optimal_disc}. Let
	\[
	(p^\star,q^\star)\in\arg\max_{\{p,q\}\in\mathcal{E}_{\mathrm{ch}}}\Delta_{p,q}(w),
	\qquad
	\Delta^\star(w):=\Delta_{p^\star,q^\star}(w).
	\]
	Assume in addition that there exists $c_0\in(0,1]$ such that
	\begin{equation}\label{eq:ratio_nondegenerate}
		\Delta^\star(w)\ge c_0(\lambda_2-\lambda_1).
	\end{equation}
	Then
	\begin{equation}\label{eq:main_optimal_disc_ratio}
		1-C_{\kappa,C_w,c_0}
		\left[
		\left(\frac{\zeta}{S}\right)^2+\sqrt{\delta_n}
		\right]
		\le
		\frac{\Delta_{\widehat p,\widehat q}(w)}{\Delta^\star(w)}
		\le 1.
	\end{equation}
	For the native choice $\zeta=r_{\max}$,
	\begin{equation}\label{eq:main_optimal_disc_ratio_native}
		1-C_{\kappa,C_w,c_0}\sqrt{\delta_n}
		\le
		\frac{\Delta_{\widehat p,\widehat q}(w)}{\Delta^\star(w)}
		\le 1.
	\end{equation}
\end{corollary}

\subsection{Probabilistic near-optimal chord selection}\label{subsec:probabilistic_discrepancy}

The deterministic theory of
Section~\ref{subsec:deterministic_discrepancy} applies once the discrepancy
parameter $\delta_n$ is sufficiently small. We now show that this occurs
with high probability under a simple i.i.d.\ bounded model. This setting is
used only as a clean baseline; the same reduction from concentration to
near-optimality extends to more general independent conductances once
comparable discrepancy estimates are available.

Assume that
\begin{equation}\label{eq:iid_kappa_assumption}
	c_0,c_1,\dots,c_{n-1}\text{ are i.i.d.\ on }[1,\kappa].
\end{equation}
Equivalently, $r_i=1/c_i$ are i.i.d.\ and lie in $[\kappa^{-1},1]$
almost surely.

The high-probability statement should be read as a moderate-heterogeneity result. For each fixed conductance ratio $\kappa$, the normalized resistance arclength coordinates become uniformly distributed around the cycle with high probability as $n$ grows. The constants depend on $\kappa$; hence the theorem does not claim uniform performance when the conductance ratio grows without bound. The severe-heterogeneity cases in Section~\ref{sec:experiments} are therefore reported as stress tests of the screening heuristic rather than as direct consequences of the theorem.

\begin{theorem}[High-probability near-optimal chord selection]\label{thm:main_optimal_random}
	Assume Eqs.~\eqref{eq:iid_kappa_assumption} and \eqref{eq:w_native_scale}.
	On each realization, choose $\zeta$ such that
	\[
	r_{\max}\le \zeta\le S/8,
	\]
	define $\mathcal A_\zeta$ accordingly, and let
	$(\widehat p,\widehat q)\in\mathcal A_\zeta$ maximize
	$|\beta_1(p,q)|$ over $\mathcal A_\zeta$. There exists
	$C_{\kappa,C_w}>0$ such that, whenever
	\[
	C_\kappa\sqrt{\frac{x+\log n}{n}}\le \delta_0,
	\qquad
	n\ge n_0,
	\]
	one has, with probability at least $1-4e^{-x}$,
	\begin{equation}\label{eq:main_random_additive}
		0\le
		\max_{\{p,q\}\in\mathcal{E}_{\mathrm{ch}}}\Delta_{p,q}(w)-\Delta_{\widehat p,\widehat q}(w)
		\le
		C_{\kappa,C_w}\frac{w}{n}
		\left[
		\left(\frac{\zeta}{S}\right)^2+
		\left(\frac{x+\log n}{n}\right)^{1/4}
		\right].
	\end{equation}
	For the native choice $\zeta=r_{\max}$,
	\begin{equation}\label{eq:main_random_additive_native}
		0\le
		\max_{\{p,q\}\in\mathcal{E}_{\mathrm{ch}}}\Delta_{p,q}(w)-\Delta_{\widehat p,\widehat q}(w)
		\le
		C_{\kappa,C_w}\frac{w}{n}
		\left(\frac{x+\log n}{n}\right)^{1/4}
	\end{equation}
	with the same probability.
\end{theorem}

\section{Pareto Front Formulation}
\label{sec:pareto}

The preceding sections provide separate tools for the two chord objectives:
$\Delta_{p,q}(w)$ for convergence-rate improvement and
$\mathcal I_{p,q}(w)$ for coherence improvement. Since both objectives are
monotone in the chord conductance under a fixed location, Pareto-efficient
solutions under a budget $0\le w\le\hat w$ use the full budget. We therefore
write, for an admissible chord $e=\{p,q\}\in\mathcal E_{\mathrm{ch}}$,
\begin{equation}\label{eq:pareto_raw_objectives}
	\Delta_e:=\Delta_{p,q}(\hat w),
	\qquad
	\mathcal I_e:=\mathcal I_{p,q}(\hat w).
\end{equation}

For a finite candidate set $\mathcal Q\subseteq\mathcal E_{\mathrm{ch}}$, the
joint design problem is the two-objective maximization
\begin{equation}\label{eq:finite_pareto_problem}
	\max_{e\in\mathcal Q}\bigl(\mathcal I_e,\Delta_e\bigr),
\end{equation}
where $\mathcal I_e$ measures Kirchhoff-index reduction and $\Delta_e$
measures algebraic-connectivity gain. A chord $e_1\in\mathcal Q$ dominates
$e_2\in\mathcal Q$, written $e_1\succ_{\mathrm P}e_2$, if
\[
\mathcal I_{e_1}\ge \mathcal I_{e_2},
\qquad
\Delta_{e_1}\ge \Delta_{e_2},
\]
and at least one inequality is strict. The Pareto-efficient set is
\begin{equation}\label{eq:pareto_efficient_set}
	\mathcal E_{\mathrm P}(\mathcal Q)
	:=
	\{e\in\mathcal Q:\nexists e'\in\mathcal Q
	\text{ such that }e'\succ_{\mathrm P}e\}.
\end{equation}

For comparison across screening rules, the objectives are normalized by the
exhaustive single-objective optima
\begin{equation}\label{eq:global_pareto_normalizers}
	\Delta^\star:=\max_{e\in\mathcal E_{\mathrm{ch}}}\Delta_e,
	\qquad
	\mathcal I^\star:=\max_{e\in\mathcal E_{\mathrm{ch}}}\mathcal I_e,
\end{equation}
and, when both normalizers are nonzero,
\begin{equation}\label{eq:normalized_pareto_objectives}
	\bar\Delta_e:=\frac{\Delta_e}{\Delta^\star},
	\qquad
	\bar{\mathcal I}_e:=\frac{\mathcal I_e}{\mathcal I^\star}.
\end{equation}
The normalized front induced by $\mathcal Q$ is
\begin{equation}\label{eq:pareto_front_image}
	\mathcal F_{\mathrm P}(\mathcal Q)
	:=
	\{(\bar{\mathcal I}_e,\bar\Delta_e):
	e\in\mathcal E_{\mathrm P}(\mathcal Q)\}.
\end{equation}
The first coordinate is coherence improvement and the second coordinate is
convergence-rate improvement.

\paragraph{Efficient coherence scoring.}
The Kirchhoff-index objective can be evaluated without recomputing a
pseudoinverse for every chord. Let $G:=L^\dagger$, $M:=G^2$, and
$\boldsymbol b_e=\boldsymbol e_p-\boldsymbol e_q$. Define
\begin{equation}\label{eq:pareto_Re_Qe_def}
	R_e:=\boldsymbol b_e^\top G\boldsymbol b_e,
	\qquad
	Q_e:=\boldsymbol b_e^\top M\boldsymbol b_e .
\end{equation}
The rank-one update formula gives
\begin{equation}\label{eq:pareto_fast_kirchhoff}
	\mathcal I_e(\hat w)
	=
	\frac{\hat w\, n Q_e}{1+\hat w R_e}.
\end{equation}
Equivalently,
\begin{equation}\label{eq:pareto_fast_kirchhoff_spectral}
	\mathcal I_e(\hat w)
	=
	\frac{\hat w n\displaystyle\sum_{k=1}^{n-1}
		\frac{(u_{k,p}-u_{k,q})^2}{\lambda_k^2}}
	{1+\hat w\displaystyle\sum_{k=1}^{n-1}
		\frac{(u_{k,p}-u_{k,q})^2}{\lambda_k}}.
\end{equation}
Thus, after forming $G$ and $M$, each candidate chord is scored for
coherence using only the four endpoint entries of each matrix.

\paragraph{Front extraction and screening.}
For a finite evaluated set $\mathcal Q$, the Pareto set can be extracted by
sorting candidates in decreasing $\bar\Delta_e$ and, with ties broken by
decreasing $\bar{\mathcal I}_e$, retaining exactly the candidates that set a
new record in $\bar{\mathcal I}_e$. This scan costs
$O(|\mathcal Q|\log |\mathcal Q|)$ including the sorting step. Taking
$\mathcal Q=\mathcal E_{\mathrm{ch}}$ gives the exhaustive front. More
generally, if $\mathcal S\subseteq\mathcal E_{\mathrm{ch}}$ is a screened
candidate set, then $\mathcal F_{\mathrm P}(\mathcal S)$ is the
$\mathcal S$-screened Pareto front.
\begin{remark}[Full and screened Pareto fronts]
	\label{rem:screened_pareto_fronts}
	The screened-front construction is independent of the screening rule itself.
	In Section~\ref{sec:experiments}, it is instantiated using the RBAPS and
	AW-RBAPS candidate sets generated by Algorithm~\ref{alg:rbaps_awrbaps}.
\end{remark}

For visualization, we also report a representative knee chord on the
exhaustive front:
\begin{equation}\label{eq:pareto_knee_chord}
	e_{\mathrm{knee}}
	\in
	\arg\min_{e\in\mathcal E_{\mathrm P}(\mathcal E_{\mathrm{ch}})}
	\sqrt{(1-\bar{\mathcal I}_e)^2+(1-\bar\Delta_e)^2}.
\end{equation}
This scalar summary is not used to define Pareto efficiency; it only
identifies the front point closest to the ideal point $(1,1)$ in the
normalized objective plane.

\section{Experiments}\label{sec:experiments}

We evaluate the proposed resistance-balanced screening rules on random
weighted cycles. The numerical study has three goals. First, we test whether
RBAPS and AW-RBAPS recover chords that are nearly optimal for the
low-frequency algebraic-connectivity objective used in the implementation.
Second, we compare this behavior with the standard Fiedler-vector endpoint
heuristic, which is a strong spectral baseline. Third, we verify that
algebraic-connectivity gain and Kirchhoff-index reduction are related but
nonidentical objectives, and then evaluate AW-RBAPS as a screened
Pareto-front method. Throughout, the chord conductance uses the full budget
$\hat w=\max_{0\le i<n}c_i$, as justified by
Remark~\ref{remark:budget}.

For a weighted cycle with conductances $c_0,\ldots,c_{n-1}$, let
\[
\mathcal E_{\rm ch}
=
\bigl\{\{p,q\}: \min\{|p-q|,n-|p-q|\}\ge 2\bigr\}
\]
be the admissible chord set. In the single-objective screening experiments,
we use the same low-frequency rank-one update both to choose a chord within
each screened set and to define the exhaustive reference optimum. Let
$0=\lambda_0<\lambda_1\le\lambda_2\le\cdots$ be the eigenvalues of the
original weighted-cycle Laplacian, and let $\boldsymbol u_1,\ldots,
\boldsymbol u_m$ be the first $m$ nontrivial eigenvectors, with
$m=12$ in all reported experiments. For $e=\{p,q\}$, set
\[
a_{pq}
=
\bigl(u_{1,p}-u_{1,q},\ldots,u_{m,p}-u_{m,q}\bigr)^\top .
\]
The implemented low-frequency approximation of the algebraic-connectivity
gain is
\begin{equation}\label{eq:lf_delta_experiment}
	\widehat\Delta_m(p,q)
	=
	\lambda_{\min}\!\left(
	\operatorname{diag}(\lambda_1,\ldots,\lambda_m)
	+
	\hat w\,a_{pq}a_{pq}^{\top}
	\right)
	-
	\lambda_1 .
\end{equation}
The reported single-objective normalized gain is
\begin{equation}\label{eq:normalized_lf_gain_experiment}
	\widehat\vartheta(e)
	=
	\frac{\widehat\Delta_m(e)}
	{\max_{f\in\mathcal E_{\rm ch}}\widehat\Delta_m(f)}.
\end{equation}
Thus $\widehat\vartheta=1$ means that the selected chord is optimal under
the same low-frequency objective among all admissible chords. The later
objective-correlation and Pareto-front experiments use exact
algebraic-connectivity gains and exact Kirchhoff-index reductions, as stated
there.

\subsection{Resistance-balanced screening and baselines}

The screening rules use resistance arclength. Let $r_i=c_i^{-1}$, build
lifted prefix sums $\tilde s_0=0$ and
$\tilde s_{k+1}=\tilde s_k+r_{k\bmod n}$ for $k=0,\ldots,2n-1$, and set
$S=\tilde s_n$. For a lifted index $k$, define $q(k)=k\bmod n$. The
subroutine $\textsc{AddPair}(i,k,\mathcal P)$ inserts
$(\min\{i,q(k)\},\max\{i,q(k)\})$ into $\mathcal P$ whenever the two
vertices are nonadjacent on the cycle. Algorithm~\ref{alg:rbaps_awrbaps}
returns the RBAPS candidate set $\mathcal P_0$ when $\tau=0$, and the
AW-RBAPS candidate set $\mathcal P_\tau$ when $\tau>0$. Unless otherwise
specified, AW-RBAPS uses $\tau=0.1$.

\begin{algorithm}[!t]
	\caption{Resistance-balanced screening: RBAPS and AW-RBAPS}
	\label{alg:rbaps_awrbaps}
	\begin{algorithmic}[1]
		\Require Cycle conductances $\{c_i\}_{i=0}^{n-1}$ and tolerance $\tau\ge0$
		\Ensure Screened candidate set $\mathcal P$
		\State $r_i\gets c_i^{-1}$ for $i=0,\ldots,n-1$
		\State Build lifted prefix sums $\tilde s_0=0$,
		$\tilde s_{k+1}=\tilde s_k+r_{k\bmod n}$, $k=0,\ldots,2n-1$
		\State $S\gets\tilde s_n$, $\mathcal P\gets\emptyset$
		\For{$i=0$ to $n-1$}
		\State $t\gets\tilde s_i+S/2$
		\State $j\gets\min\{k\in\{i+1,\ldots,i+n\}:\tilde s_k\ge t\}$
		\ForAll{$k_0\in\{j-1,j,j+1\}\cap\{i+1,\ldots,i+n-1\}$}
		\State $\textsc{AddPair}(i,k_0,\mathcal P)$
		\If{$\tau>0$}
		\State expand left and right from $k_0$ while
		$|2(\tilde s_k-\tilde s_i)-S|\le \tau S$, inserting admissible pairs
		\EndIf
		\EndFor
		\EndFor
		\State \Return $\mathcal P$
	\end{algorithmic}
\end{algorithm}

We compare four strategies. Random selects one chord uniformly from
$\mathcal E_{\rm ch}$. The Fiedler baseline first takes the max--min
endpoint pair
\begin{equation}\label{eq:fiedler_rule_experiment}
	e_{\rm F}^{0}
	=
	\left\{
	\arg\min_i u_{1,i},\,
	\arg\max_i u_{1,i}
	\right\}.
\end{equation}
If $e_{\rm F}^{0}\in\mathcal E_{\rm ch}$, this chord is used. If the two
endpoints are adjacent, we use the admissible fallback
\[
e_{\rm F}\in
\arg\max_{\{p,q\}\in\mathcal E_{\rm ch}}
|u_{1,p}-u_{1,q}|^2 .
\]
This is a strong benchmark because the squared Fiedler-vector gap is the
first-order perturbation proxy for algebraic-connectivity improvement.

RBAPS and AW-RBAPS first generate $\mathcal P_0$ and $\mathcal P_\tau$,
respectively, and then select
\[
e_{\rm R}\in\arg\max_{e\in\mathcal P_0}\widehat\Delta_m(e),
\qquad
e_{\rm AW}\in\arg\max_{e\in\mathcal P_\tau}\widehat\Delta_m(e).
\]
Thus the comparison is between a single-mode Fiedler endpoint heuristic and
a resistance-screened, multi-mode low-frequency update.

\paragraph{Computational distinction.}
The resistance-screening stage is eigenvector-free: it uses only cycle-edge
resistances and prefix sums. The optional score
\eqref{eq:lf_delta_experiment} requires the first $m$ nontrivial
eigenvectors once, after which only the screened candidate set is scored.
The computational point is therefore not that Fiedler endpoint search is
intrinsically quadratic. After the Fiedler vector is known, an optimized
endpoint implementation can be found by scanning the largest and smallest
components. The advantage of RBAPS/AW-RBAPS is instead that they provide a
cycle-structured candidate set that can reduce subsequent low-frequency,
exact spectral, or Pareto objective evaluations from the full
$\Theta(n^2)$ chord set to a linear or near-linear set.

\subsection{Baseline comparison for algebraic-connectivity screening}
\label{subsec:algconn_experiments}

The default setting is $n=200$, independent conductances sampled uniformly
from $[1,100]$, and $\hat w=100$. We repeat the experiment over four
independent rounds, each with $1000$ weighted cycles.

\begin{figure}[t!]
	\centering
	\includegraphics[width=0.6\columnwidth]{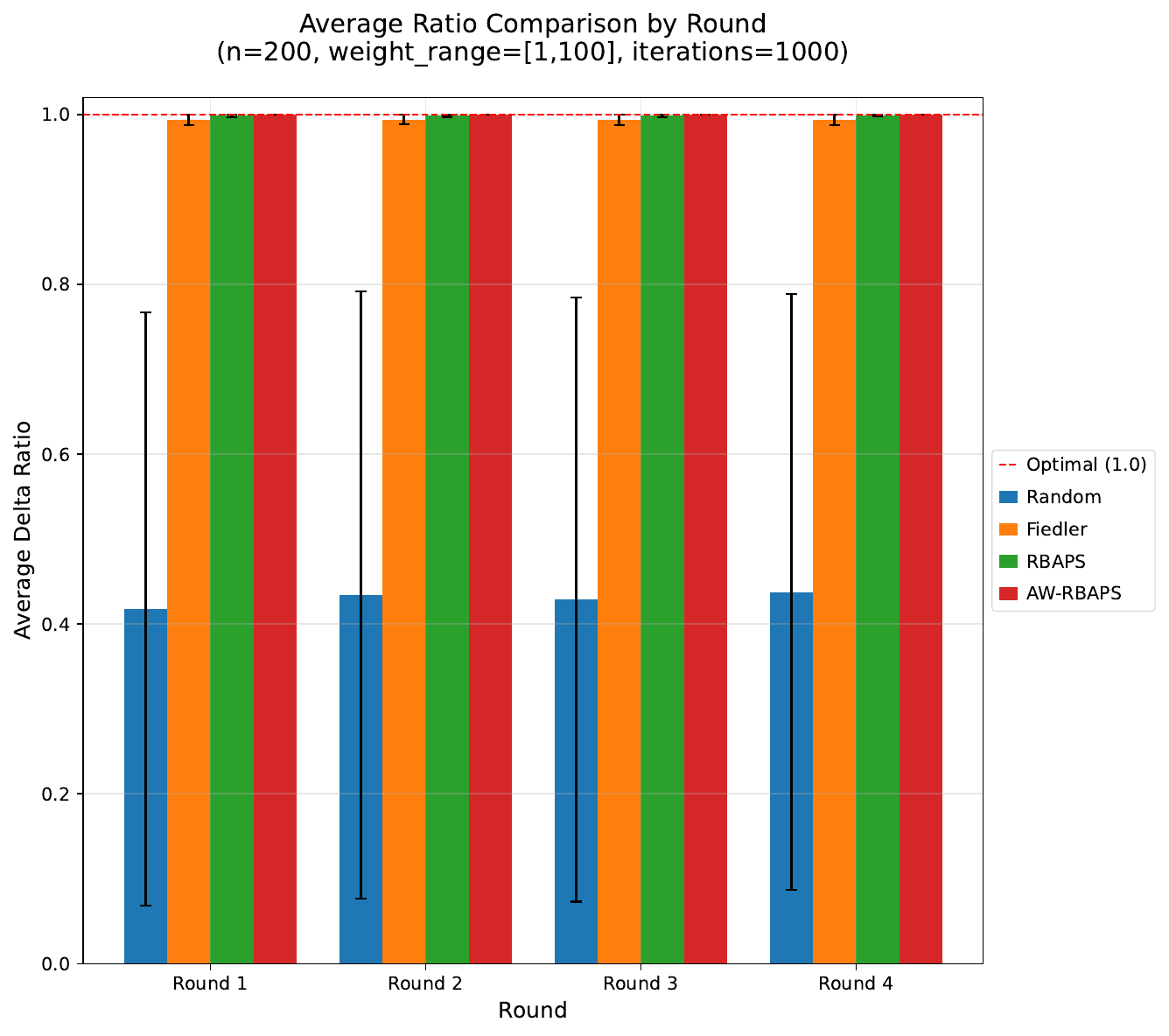}
	\caption{Average normalized low-frequency gain over four independent rounds
		in the default setting. Fiedler is already a strong baseline, while RBAPS and
		especially AW-RBAPS further reduce the remaining optimality gap.}
	\label{fig:baseline_strategy}
\end{figure}

Fig.~\ref{fig:baseline_strategy} shows that random chord insertion is far
from optimal and highly variable, while all informed strategies are close to
the exhaustive low-frequency optimum. Averaged over all $4000$ trials, the
mean normalized gains are approximately $0.4296$ for Random, $0.9939$
for the Fiedler heuristic, $0.9986$ for RBAPS, and $0.9998$ for
AW-RBAPS. These values show that the Fiedler endpoint rule is not a weak
baseline: it already captures most of the benefit predicted by first-order
spectral perturbation. The proposed resistance-balanced rules improve on
this strong baseline by shrinking the residual optimality gap. Relative to
Fiedler, RBAPS reduces this gap by roughly three quarters, and AW-RBAPS
reduces it by more than an order of magnitude.

\begin{figure*}[htbp]
	\centering
	\begin{minipage}{0.48\textwidth}
		\centering
		\includegraphics[width=\linewidth]{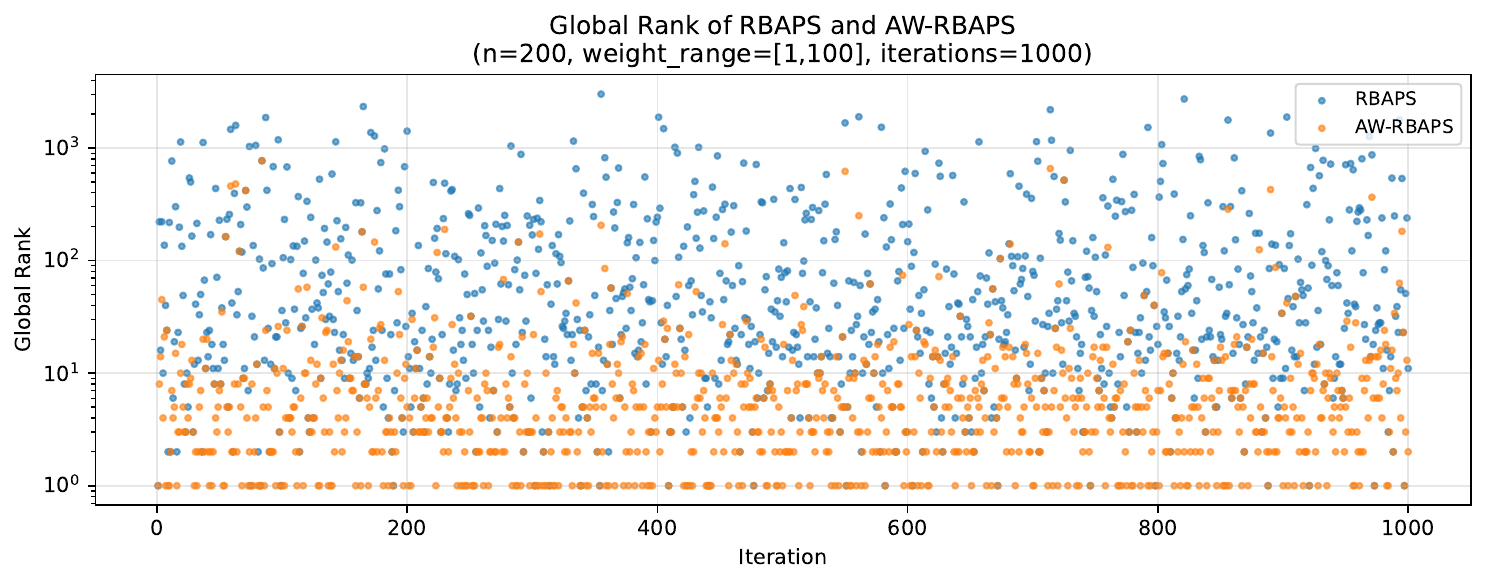}
	\end{minipage}\hfill
	\begin{minipage}{0.48\textwidth}
		\centering
		\includegraphics[width=\linewidth]{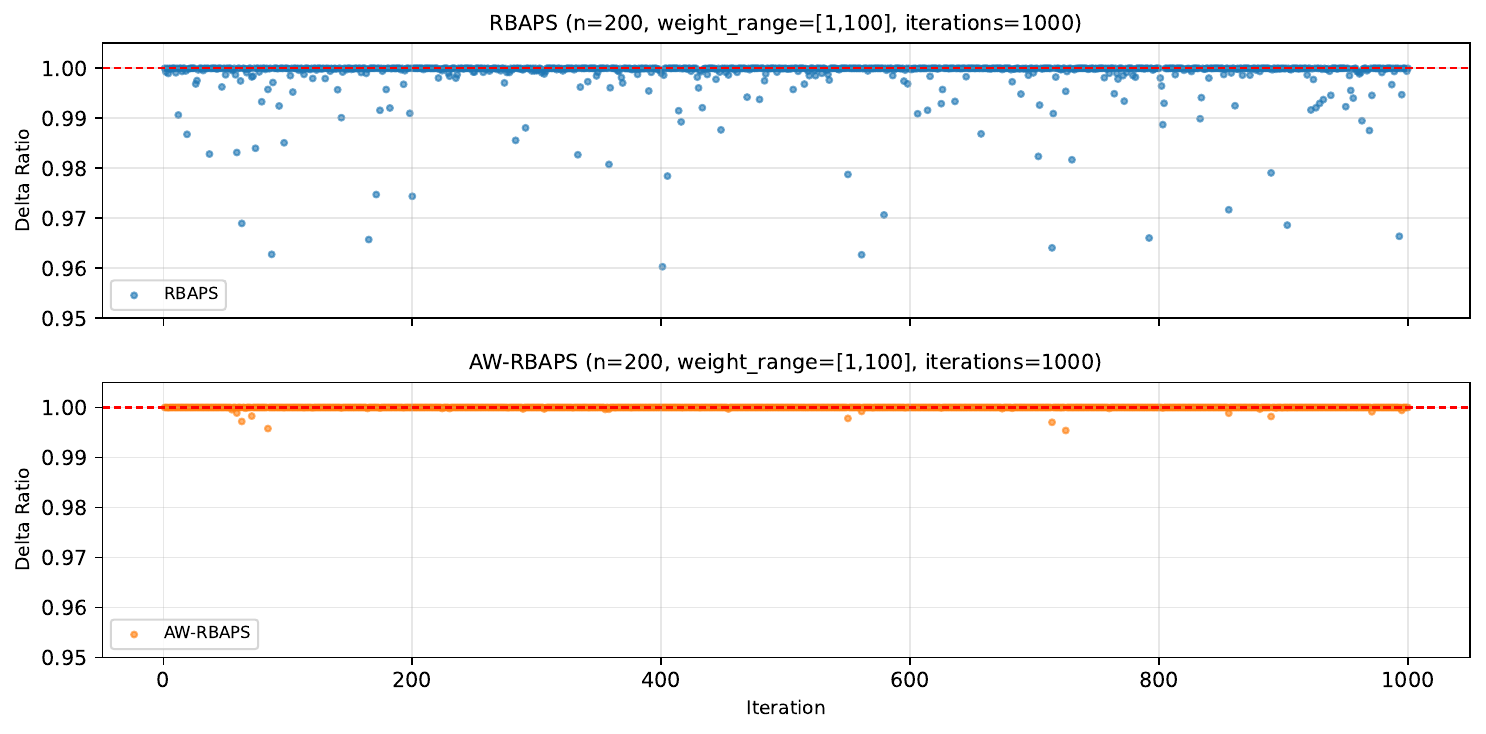}
	\end{minipage}
	\caption{Left: global rank, under the exhaustive low-frequency objective, of
		the best edge screened by RBAPS and AW-RBAPS. Right: corresponding normalized
		low-frequency gains. AW-RBAPS is more concentrated near the global optimum,
		while both resistance-balanced methods remain close to optimal in objective
		value.}
	\label{fig:rank_ratio}
\end{figure*}

Fig.~\ref{fig:rank_ratio} gives a stricter diagnostic than average gain.
RBAPS can select screened edges with less favorable exhaustive rank, whereas
AW-RBAPS is more concentrated near the best augmentation. The right panel
explains why both methods still have high normalized gain: many near-antipodal
chords are almost tied under $\widehat\Delta_m$. This supports the
cycle-specific intuition that resistance arclength identifies the relevant
high-quality region, and that the adaptive window stabilizes the final
selection within that region.

\subsection{Effect of conductance heterogeneity and graph size}

We next test robustness. First, with $n=200$, the conductance range is
varied from $[1,100]$ to $[1,10^8]$, with
$\hat w$ set to the upper endpoint of the range. Second, the conductance
range is fixed to $[1,100]$, and $n$ is varied from $200$ to $800$.
Each setting uses $1000$ independent instances.

\begin{table*}[!t]
	\centering
	\caption{Average normalized low-frequency gain $\widehat\vartheta$ under
		different conductance ranges ($n=200$, $1000$ trials per setting). Each
		entry is mean $\pm$ standard deviation.}
	\label{tab:weight_range}
	\renewcommand{\arraystretch}{1.05}
	\begin{tabular}{lcccc}
		\hline
		Conductance range & Random & Fiedler & RBAPS & AW-RBAPS \\
		\hline
		$[1,100]$      & $0.4356 \pm 0.3478$ & $0.9939 \pm 0.0085$ & $0.9986 \pm 0.0063$ & $\mathbf{0.9999 \pm 0.0010}$ \\
		$[1,10^4]$     & $0.3747 \pm 0.3466$ & $0.9778 \pm 0.0478$ & $0.9796 \pm 0.0648$ & $\mathbf{0.9968 \pm 0.0206}$ \\
		$[1,10^6]$     & $0.3651 \pm 0.3463$ & $0.9734 \pm 0.0581$ & $0.9798 \pm 0.0650$ & $\mathbf{0.9970 \pm 0.0163}$ \\
		$[1,10^8]$     & $0.3766 \pm 0.3459$ & $0.9718 \pm 0.0601$ & $0.9726 \pm 0.0821$ & $\mathbf{0.9947 \pm 0.0289}$ \\
		\hline
	\end{tabular}
\end{table*}

\begin{table*}[!t]
	\centering
	\caption{Average normalized low-frequency gain $\widehat\vartheta$ under
		different numbers of vertices (conductance range $[1,100]$, $1000$
		trials per setting). Each entry is mean $\pm$ standard deviation. Entries
		displayed as $1.0000\pm0.0000$ are rounded to four decimal places.}
	\label{tab:n_vertices}
	\renewcommand{\arraystretch}{1.05}
	\begin{tabular}{lcccc}
		\hline
		Number of vertices & Random & Fiedler & RBAPS & AW-RBAPS \\
		\hline
		$200$ & $0.4332 \pm 0.3562$ & $0.9938 \pm 0.0086$ & $0.9983 \pm 0.0068$ & $\mathbf{0.9999 \pm 0.0008}$ \\
		$400$ & $0.4476 \pm 0.3559$ & $0.9966 \pm 0.0049$ & $0.9996 \pm 0.0015$ & $\mathbf{1.0000 \pm 0.0000}$ \\
		$600$ & $0.4719 \pm 0.3506$ & $0.9975 \pm 0.0034$ & $0.9998 \pm 0.0008$ & $\mathbf{1.0000 \pm 0.0000}$ \\
		$800$ & $0.4429 \pm 0.3473$ & $0.9981 \pm 0.0024$ & $0.9999 \pm 0.0005$ & $\mathbf{1.0000 \pm 0.0000}$ \\
		\hline
	\end{tabular}
\end{table*}

\begin{figure*}[!t]
	\centering
	\begin{minipage}{0.48\textwidth}
		\centering
		\includegraphics[width=\linewidth]{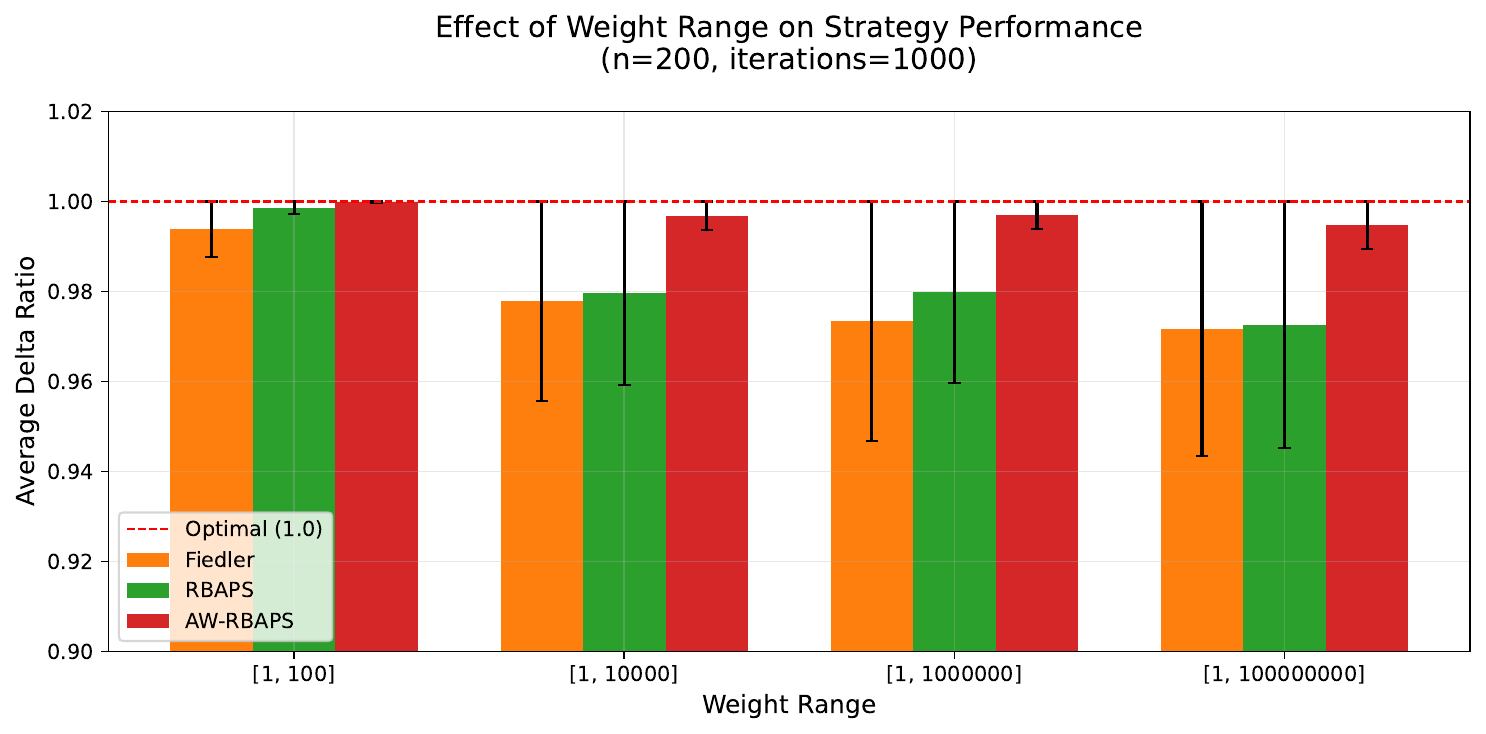}
	\end{minipage}\hfill
	\begin{minipage}{0.48\textwidth}
		\centering
		\includegraphics[width=\linewidth]{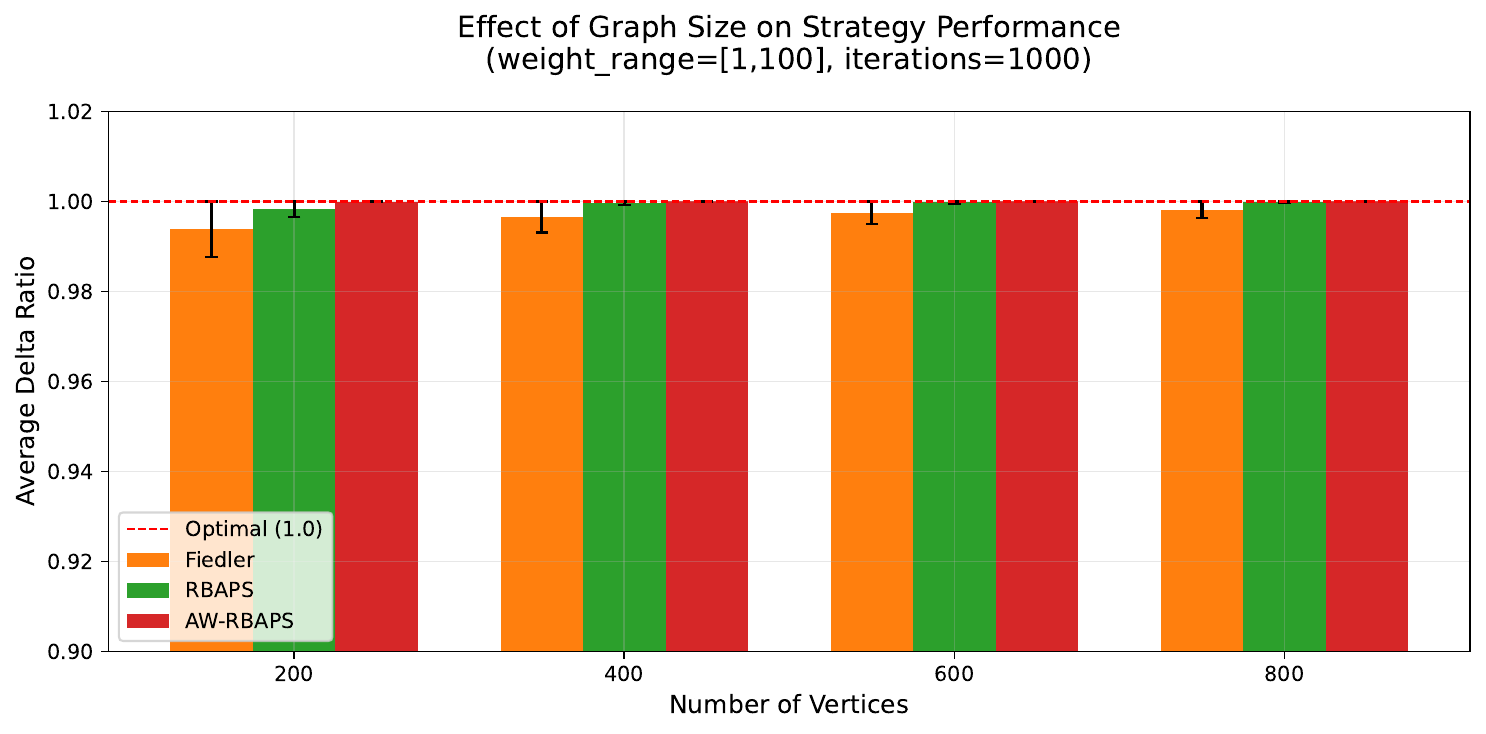}
	\end{minipage}
	\caption{Effect of conductance heterogeneity and graph size. Left: increasing
		heterogeneity degrades informed methods, but AW-RBAPS remains closest to the
		exhaustive low-frequency optimum. Right: under the moderate range
		$[1,100]$, RBAPS and AW-RBAPS improve with graph size.}
	\label{fig:expansion_visual}
\end{figure*}

Tables~\ref{tab:weight_range}--\ref{tab:n_vertices} and
Fig.~\ref{fig:expansion_visual} support two main conclusions. First,
AW-RBAPS is the most robust informed method under large conductance
heterogeneity. Even for $[1,10^8]$, its average normalized gain remains
above $0.994$, whereas the Fiedler baseline decreases to about $0.972$.
This is the clearest numerical advantage of the adaptive resistance window:
when the weighted cycle is highly nonuniform, retaining a local family of
near-resistance-balanced alternatives is more stable than relying on a
single Fiedler endpoint pair. Second, for fixed moderate heterogeneity
$[1,100]$, the resistance-balanced methods improve with $n$. AW-RBAPS is
essentially optimal at the reported precision for $n=400,600,800$, and
RBAPS also approaches the exhaustive low-frequency optimum. This trend is
consistent with the high-probability discrepancy scaling in
Theorem~\ref{thm:discrepancy_iid_kappa}; the extreme $[1,10^8]$ cases are
best interpreted as stress tests beyond the moderate-heterogeneity regime.

\subsection{Algebraic-connectivity gain versus Kirchhoff-index reduction}
\label{subsec:objective_relation}

The previous experiments focus on the low-frequency algebraic-connectivity
screening objective. We now compare the exact algebraic-connectivity gain
with the exact Kirchhoff-index reduction. For each weighted cycle in the
default setting, $\Delta_e$ is computed from the eigenvalue change
$\lambda_1(L_e(\hat w))-\lambda_1(L)$, while $\mathcal I_e$ is evaluated
using the exact rank-one formula from Section~\ref{sec:resistance}.

\begin{figure}[!t]
	\centering
	\includegraphics[width=0.6\columnwidth]{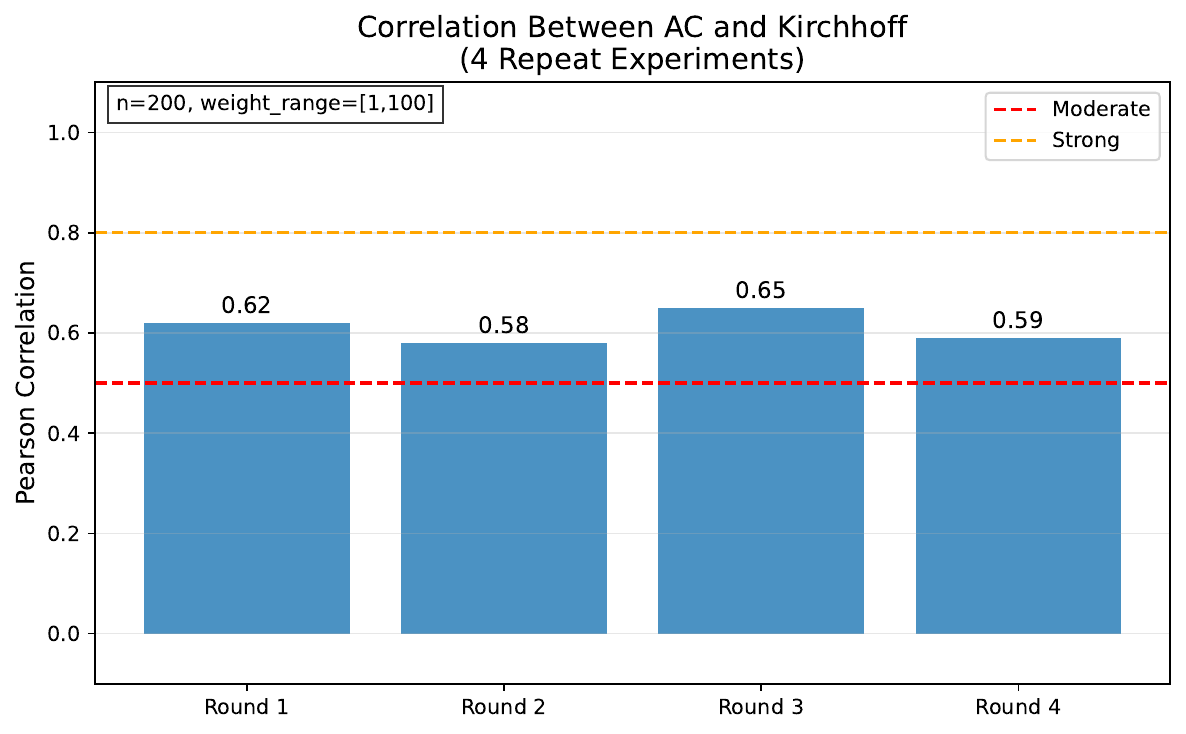}
	\caption{Correlation between exact algebraic-connectivity gain and exact
		Kirchhoff-index reduction over four repeated experiments in the default
		setting.}
	\label{fig:ac_kirchhoff_corr}
\end{figure}

Fig.~\ref{fig:ac_kirchhoff_corr} shows that the two objectives are positively
correlated but not equivalent. Across the four repeated experiments, the
Pearson correlation lies between $0.58$ and $0.65$, indicating only
moderate alignment. In particular, the chord that maximizes
algebraic-connectivity gain does not generally coincide with the chord that
maximizes Kirchhoff-index reduction. Hence the two criteria capture related
but distinct aspects of network enhancement, which motivates the finite
Pareto-front formulation in Section~\ref{sec:pareto}.

\subsection{Pareto-front experiment}
\label{subsec:pareto_experiment}

The Pareto experiment uses the same graph-generation model and chord budget
as the default setting, but evaluates both exact objectives for each
candidate chord. For a generated weighted cycle, define
\[
\mathcal Q_{\rm full}:=\mathcal E_{\rm ch},
\qquad
\mathcal Q_{\rm R}:=\mathcal P_0,
\qquad
\mathcal Q_{\rm AW}:=\mathcal P_\tau .
\]
By Remark~\ref{rem:screened_pareto_fronts},
$\mathcal F_{\rm P}(\mathcal Q_{\rm R})$ and
$\mathcal F_{\rm P}(\mathcal Q_{\rm AW})$ are the RBAPS and AW-RBAPS
screened Pareto fronts. Since AW-RBAPS is consistently strongest in the
single-objective tests, we report its Pareto-front performance.

For each $e\in\mathcal Q_{\rm full}$, the algebraic-connectivity gain
$\Delta_e$ is computed from $\lambda_1(L_e(\hat w))$, and the
Kirchhoff-index reduction $\mathcal I_e$ is computed using
\eqref{eq:pareto_fast_kirchhoff}. The exhaustive front
$\mathcal F_{\rm P}(\mathcal Q_{\rm full})$ serves as the reference, and
all fronts are compared in the normalized objective plane
\eqref{eq:normalized_pareto_objectives}.

We use three diagnostics. The exact screened-front coverage is
\begin{equation}\label{eq:pareto_coverage_metric}
	\operatorname{cov}(\mathcal Q)
	:=
	\frac{|\mathcal E_{\rm P}(\mathcal Q_{\rm full})\cap\mathcal Q|}
	{|\mathcal E_{\rm P}(\mathcal Q_{\rm full})|}.
\end{equation}
The additive $\epsilon$-dominance error is
\begin{equation}\label{eq:pareto_epsilon_metric}
	\epsilon_+(\mathcal Q)
	:=
	\max_{e\in\mathcal E_{\rm P}(\mathcal Q_{\rm full})}
	\min_{e'\in\mathcal E_{\rm P}(\mathcal Q)}
	\max\{\bar{\mathcal I}_e-\bar{\mathcal I}_{e'},\,
	\bar\Delta_e-\bar\Delta_{e'},\,0\}.
\end{equation}
Finally, with reference point $(0,0)$, let $\operatorname{HV}(\mathcal Q)$
be the dominated hypervolume of $\mathcal F_{\rm P}(\mathcal Q)$, and
report
\begin{equation}\label{eq:pareto_hv_metric}
	h_{\rm HV}(\mathcal Q)
	:=
	\frac{\operatorname{HV}(\mathcal Q)}
	{\operatorname{HV}(\mathcal Q_{\rm full})} .
\end{equation}
Larger coverage and hypervolume ratio, and smaller $\epsilon_+$, indicate
a better screened approximation of the exhaustive front.

We first report one representative instance with $n=200$, conductances
sampled from $[1,100]$, $\hat w=100$, $\tau=0.1$, and random seed
$2026$. The exhaustive search contains $19700$ admissible chords, while
AW-RBAPS generates $2030$ candidates, about $10.3\%$ of the exhaustive
set.

\begin{figure*}[!t]
	\centering
	\includegraphics[width=.92\textwidth]{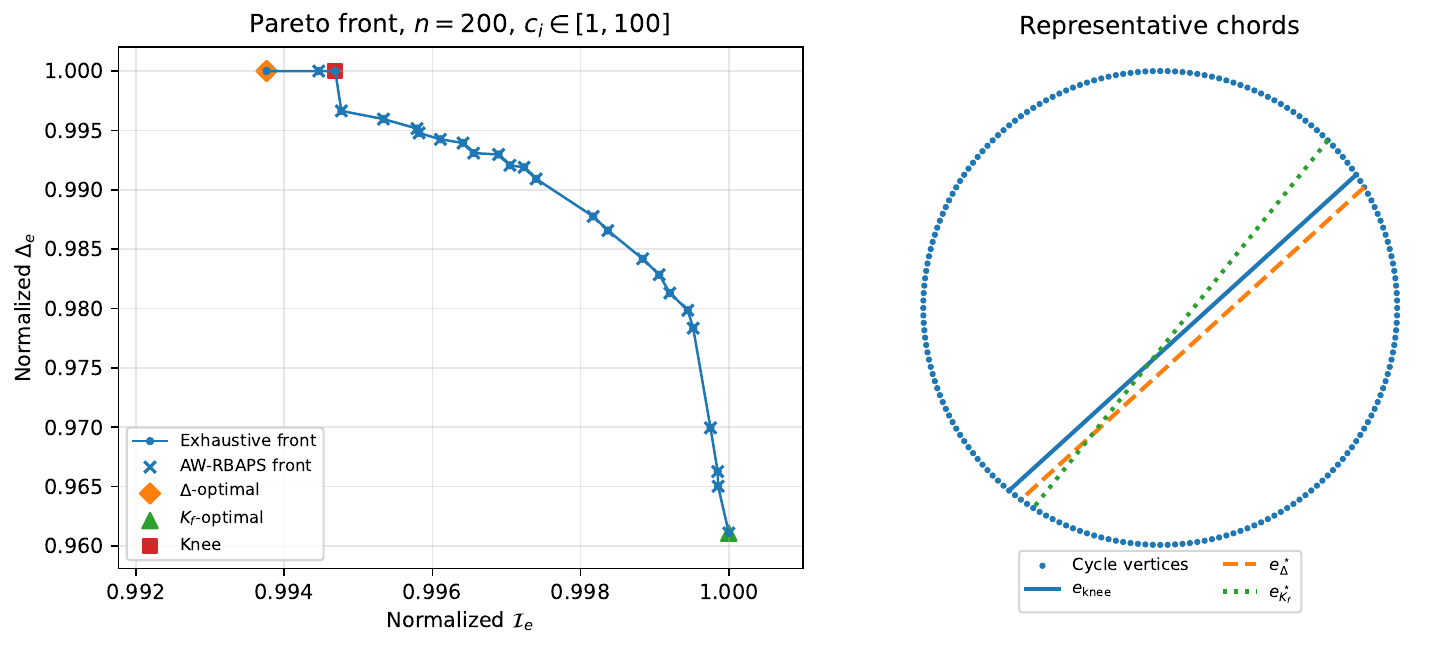}
	\caption{Pareto-front example for $n=200$ and conductance range
		$[1,100]$. Left: exhaustive and AW-RBAPS screened fronts in the normalized
		objective plane. The AW-RBAPS front coincides with the exhaustive front in
		this instance. Right: representative chords for the knee point, the
		algebraic-connectivity optimum, and the Kirchhoff-index optimum.}
	\label{fig:pareto_fronts_aw}
\end{figure*}

In addition, we perform a $100$-run Monte Carlo evaluation under the same
setting. In every run, $|\mathcal E_{\rm ch}|=19700$, while AW-RBAPS
evaluates only about one tenth of all admissible chords. The candidate-set
ratio has mean $0.1012$, standard deviation $0.0078$, median $0.1012$,
and range $[0.0771,0.1228]$, corresponding to about $1993$ candidates on
average.

\begin{figure*}[!t]
	\centering
	\includegraphics[width=1\textwidth]{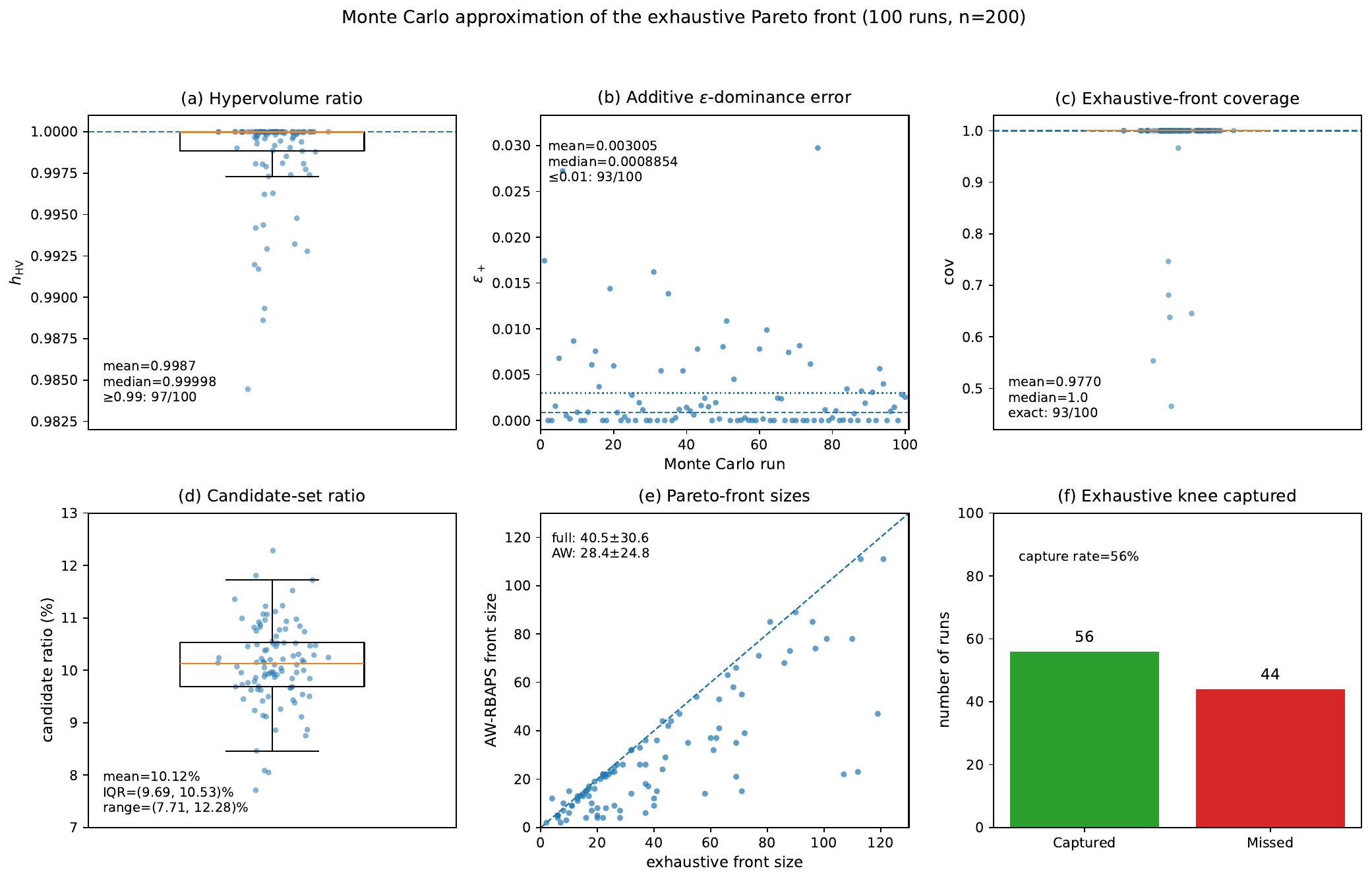}
	\caption{Monte Carlo Pareto-front approximation performance of AW-RBAPS over
		$100$ independent $n=200$ weighted-cycle instances. AW-RBAPS uses about
		$10.1\%$ of all admissible chords on average while preserving an almost
		identical dominated hypervolume in most runs.}
	\label{fig:pareto_metrics_aw}
\end{figure*}

Fig.~\ref{fig:pareto_fronts_aw} illustrates a typical front-level outcome:
the algebraic-connectivity optimum, Kirchhoff-index optimum, and knee chord
need not coincide, but the AW-RBAPS screened set retains the high-quality
trade-off region. The Monte Carlo statistics in
Fig.~\ref{fig:pareto_metrics_aw} confirm that this behavior is stable. The
mean hypervolume ratio is $0.9987$, the median is $0.99998$, and $97$
out of $100$ runs have $h_{\rm HV}\ge0.99$. The additive
$\epsilon$-dominance error is also small: its mean is
$3.00\times10^{-3}$, its median is $8.85\times10^{-4}$, and $93$ out
of $100$ runs have $\epsilon_+\le10^{-2}$. These metrics show that the
screened front is almost indistinguishable from the exhaustive front in
hypervolume and additive-dominance terms.

Exact coverage is stricter because it counts only exhaustive Pareto points
that are literally retained in the screened candidate set. Even under this
criterion, AW-RBAPS achieves exact coverage in $93$ out of $100$ runs,
with mean coverage $0.9770$ and median coverage $1$. The minimum coverage
is $0.4653$, but this worst case still has $h_{\rm HV}=0.9979$, showing
that missed exhaustive-front points are often replaced by nearby screened
alternatives. The exhaustive front contains $40.48\pm30.62$ points on
average, whereas the AW-RBAPS screened front contains $28.36\pm24.85$
points on average. Finally, the exact knee chord is captured in $56$ out
of $100$ runs. This single-point recovery criterion is more sensitive than
front-level metrics, so missing the exact knee has little effect when nearby
Pareto-efficient alternatives are retained.

Overall, the experiments show that AW-RBAPS is not merely a favorable
single-instance heuristic. It consistently preserves the high-quality portion
of the exhaustive Pareto front while evaluating only a near-linear screened
candidate set.

\section{Conclusion}\label{sec:conclusion}

This paper studied one-chord augmentation of weighted cycles for noisy first-order consensus. The main message is that a chord in a weighted cycle should be understood through the complementary resistance arcs that it creates. This geometry complements graph-generic edge-addition viewpoints: the same resistance split enters the exact coherence update and the low-frequency spectral mechanism that governs algebraic-connectivity improvement.

The first part of the analysis derived exact effective-resistance and Kirchhoff-index formulas after adding a chord, which gives a closed-form objective for reducing first-order network coherence. The second part used a cycle-adapted secular equation and a two-mode comparison theorem to show that near-antipodal resistance-balanced chords are near-optimal for algebraic-connectivity gain under bounded conductances and small discrepancy, with a high-probability extension for an i.i.d. bounded-conductance model. The third part converted these insights into the RBAPS and AW-RBAPS screening rules and into a finite Pareto-front formulation for the joint convergence-rate/coherence trade-off.

The experiments support the proposed design principle. The Fiedler endpoint rule is already a strong graph-generic baseline, but RBAPS and especially AW-RBAPS reduce the remaining gap to the exhaustive low-frequency optimum. For exact objectives, algebraic-connectivity gain and Kirchhoff-index reduction are positively correlated but not equivalent, which justifies the Pareto formulation. In the reported $n=200$ Monte Carlo Pareto tests, AW-RBAPS evaluates about $10.1\%$ of all admissible chords while attaining mean hypervolume ratio $0.9987$ relative to the exhaustive front. The severe-heterogeneity experiments further suggest that the resistance-balanced candidate set can remain useful beyond the current proof regime.

Future work will extend the framework to multiple chords, directed or switching communication cycles, and UAV formation models in which topology augmentation is coupled with motion constraints, communication range, and energy-aware scheduling.

\section*{Acknowledgement(s)}
The authors would like to express their sincere gratitude to all the referees for their careful reading and insightful suggestions.

\appendix
\section{Auxiliary results for Section~\ref{sec:algconn}}
\label{app:auxiliary-results}
This appendix contains the deterministic discrepancy and probabilistic concentration lemmas used to prove the near-antipodal resistance-balance theorems in Section~\ref{sec:algconn}.

\subsection{Quadrature and continuous embedding}

\begin{lemma}[Periodic quadrature under node discrepancy]\label{lem:quadrature_discrepancy}
	If $g\in C^1([0,1])$ and $g(0)=g(1)$, then
	\[
	\left|
	\frac1n\sum_{i=0}^{n-1}g(y_i)-\int_0^1g(t)\,\mathrm dt
	\right|
	\le
	\left(\Delta+\frac1n\right)\int_0^1|g'(t)|\,\mathrm dt.
	\]
\end{lemma}

\begin{proof}
	Let $\nu_n$ be the empirical measure of $\{y_i\}_{i=0}^{n-1}$, and
	let
	\[
	F_n(t):=\nu_n([0,t]),
	\qquad
	R_n(t):=F_n(t)-t.
	\]
	The discrepancy bound \eqref{eq:node_discrepancy} implies
	\[
	\|R_n\|_\infty\le \Delta+\frac1n.
	\]
	Indeed, if $k=\#\{0\le i\le n-1:y_i\le t\}$, then
	$F_n(t)=k/n$, and the mesh discrepancy between $y_k$ and $k/n$
	yields the estimate.
	
	Now
	\[
	\frac1n\sum_{i=0}^{n-1}g(y_i)-\int_0^1g(t)\,\mathrm dt
	=
	\int_{[0,1]}g(t)\,d(\nu_n-\lambda)(t),
	\]
	where $\lambda$ is Lebesgue measure. Since
	$(\nu_n-\lambda)([0,1])=0$, subtracting the constant $g(0)$ and
	integrating by parts gives
	\[
	\int_{[0,1]}g(t)\,d(\nu_n-\lambda)(t)
	=
	-\int_0^1 R_n(s)g'(s)\,ds.
	\]
	Hence
	\[
	\left|
	\frac1n\sum_{i=0}^{n-1}g(y_i)-\int_0^1g(t)\,\mathrm dt
	\right|
	\le
	\|R_n\|_\infty\int_0^1|g'(s)|\,ds,
	\]
	and the claim follows.
\end{proof}

\begin{lemma}[Energy and mass equivalence]\label{lem:continuous_embed_disc}
	Let $x_n:=x_0$ and $s_n:=S$. For $\boldsymbol{x}\in\mathbb{R}^n$,
	let $f_{\boldsymbol{x}}:[0,S]\to\mathbb{R}$ be the continuous periodic
	piecewise-linear interpolant defined by
	\[
	f_{\boldsymbol{x}}(s_i)=x_i,
	\qquad
	f_{\boldsymbol{x}}(t)=x_i+\frac{t-s_i}{r_i}(x_{i+1}-x_i),
	\quad t\in[s_i,s_{i+1}],
	\quad i=0,\dots,n-1.
	\]
	Then
	\begin{equation}\label{eq:energy_identity}
		\boldsymbol{x}^\top L\boldsymbol{x}
		=
		\int_0^S|f_{\boldsymbol{x}}'(t)|^2\,\mathrm dt.
	\end{equation}
	Moreover, there exist constants $m_1,m_2>0$, depending only on
	$\kappa$, such that
	\begin{equation}\label{eq:mass_equivalence}
		m_1\bar r\|\boldsymbol{x}\|_2^2
		\le \|f_{\boldsymbol{x}}\|_{L^2(0,S)}^2
		\le m_2\bar r\|\boldsymbol{x}\|_2^2.
	\end{equation}
\end{lemma}

\begin{proof}
	On each interval $[s_i,s_{i+1}]$,
	\[
	f_{\boldsymbol{x}}'(t)=\frac{x_{i+1}-x_i}{r_i},
	\]
	so
	\[
	\int_{s_i}^{s_{i+1}}|f_{\boldsymbol{x}}'(t)|^2\,\mathrm dt
	=
	\frac{(x_{i+1}-x_i)^2}{r_i}
	=
	c_{i,i+1}(x_i-x_{i+1})^2.
	\]
	Summing over $i$ yields \eqref{eq:energy_identity}.
	
	Further,
	\[
	\int_{s_i}^{s_{i+1}}f_{\boldsymbol{x}}(t)^2\,\mathrm dt
	=
	\frac{r_i}{3}\bigl(x_i^2+x_ix_{i+1}+x_{i+1}^2\bigr).
	\]
	Using
	\[
	\frac12(x_i^2+x_{i+1}^2)
	\le x_i^2+x_ix_{i+1}+x_{i+1}^2
	\le \frac32(x_i^2+x_{i+1}^2),
	\]
	we obtain
	\[
	\frac{r_i}{6}(x_i^2+x_{i+1}^2)
	\le
	\int_{s_i}^{s_{i+1}}f_{\boldsymbol{x}}(t)^2\,\mathrm dt
	\le
	\frac{r_i}{2}(x_i^2+x_{i+1}^2).
	\]
	Summing over $i$ gives
	\[
	\frac{1}{6}\sum_{i=0}^{n-1}(r_i+r_{i-1})x_i^2
	\le
	\|f_{\boldsymbol{x}}\|_{L^2(0,S)}^2
	\le
	\frac{1}{2}\sum_{i=0}^{n-1}(r_i+r_{i-1})x_i^2.
	\]
	Since $r_i\in[\kappa^{-1},1]$ and $\bar r\in[\kappa^{-1},1]$, this
	proves \eqref{eq:mass_equivalence}.
\end{proof}

\begin{lemma}[Continuous spectral gap on the circle]\label{lem:circle_gap_disc}
	Let
	\[
	\mu_1:=\left(\frac{2\pi}{S}\right)^2,
	\qquad
	\mathcal U:=\operatorname{span}\left\{
	\sin\frac{2\pi t}{S},\ \cos\frac{2\pi t}{S}
	\right\}\subset H^1_{\mathrm{per}}(0,S).
	\]
	If $g\in H^1_{\mathrm{per}}(0,S)$ has mean zero and
	$g_\perp:=g-P_{\mathcal U}g$, then
	\begin{equation}\label{eq:circle_gap_disc}
		\int_0^S|g'(t)|^2\,\mathrm dt
		\ge
		\mu_1\int_0^S|g(t)|^2\,\mathrm dt
		+
		3\mu_1\int_0^S|g_\perp(t)|^2\,\mathrm dt.
	\end{equation}
\end{lemma}

\begin{proof}
	Expand $g$ in the real Fourier basis on $[0,S]$. The $m=1$ modes
	have eigenvalue $\mu_1$, whereas every mode $m\ge 2$ has eigenvalue at
	least $4\mu_1$. This gives \eqref{eq:circle_gap_disc}.
\end{proof}

\begin{lemma}[Low-energy mass transfer]\label{lem:low_energy_mass_transfer}
	Let $\boldsymbol{z}\in\mathbb{R}^n$ satisfy
	\[
	\boldsymbol{z}^\top L\boldsymbol{z}\le C_*n^{-2}\|\boldsymbol{z}\|_2^2.
	\]
	Then
	\begin{equation}\label{eq:low_energy_mass_transfer}
		\|f_{\boldsymbol{z}}\|_{L^2(0,S)}^2
		=
		\bar r\|\boldsymbol{z}\|_2^2\bigl(1+O_{\kappa,C_*}(\delta_n)\bigr).
	\end{equation}
	If $\boldsymbol{z}$ and $\boldsymbol{y}$ both satisfy the same
	low-energy bound, then
	\begin{equation}\label{eq:bilinear_mass_transfer}
		\left|
		\langle f_{\boldsymbol{z}},f_{\boldsymbol{y}}\rangle_{L^2(0,S)}
		-\bar r\langle\boldsymbol{z},\boldsymbol{y}\rangle_{\mathbb{R}^n}
		\right|
		\le
		C_{\kappa,C_*}\delta_n\bar r
		\|\boldsymbol{z}\|_2\|\boldsymbol{y}\|_2.
	\end{equation}
\end{lemma}

\begin{proof}
	By homogeneity, assume first that $\|\boldsymbol{z}\|_2=1$, and write
	$f:=f_{\boldsymbol{z}}$. Define
	\[
	h(t):=\frac{S\,f(St)^2}{\|f\|_{L^2(0,S)}^2},
	\qquad t\in[0,1].
	\]
	Then $\int_0^1 h(t)\,dt=1$, and
	\[
	\frac1n\sum_{i=0}^{n-1}h(y_i)
	=
	\frac{S}{n\,\|f\|_{L^2(0,S)}^2}\sum_{i=0}^{n-1}z_i^2
	=
	\frac{\bar r}{\|f\|_{L^2(0,S)}^2}.
	\]
	Moreover,
	\[
	\int_0^1|h'(t)|\,dt
	\le
	\frac{2S\|f'\|_{L^2(0,S)}}{\|f\|_{L^2(0,S)}}
	\le C_{\kappa,C_*},
	\]
	where we used Lemma~\ref{lem:continuous_embed_disc}, the low-energy bound,
	and $S\asymp_\kappa n$. Applying
	Lemma~\ref{lem:quadrature_discrepancy} to $h$ gives
	\[
	\left|
	\frac{\bar r}{\|f\|_{L^2(0,S)}^2}-1
	\right|
	\le
	C_{\kappa,C_*}\left(\Delta+\frac1n\right).
	\]
	Since $n^{-1}\le C_\kappa \eta\le C_\kappa\delta_n$, this proves
	\eqref{eq:low_energy_mass_transfer}.
	
	For the bilinear estimate, write $f:=f_{\boldsymbol{z}}$,
	$g:=f_{\boldsymbol{y}}$, and
	\[
	q(t):=f(St)g(St),\qquad t\in[0,1].
	\]
	Then $q(0)=q(1)$,
	\[
	\frac1n\sum_{i=0}^{n-1}q(y_i)
	=
	\frac1n\langle\boldsymbol{z},\boldsymbol{y}\rangle_{\mathbb{R}^n},
	\qquad
	\int_0^1q(t)\,dt
	=
	\frac1S\langle f,g\rangle_{L^2(0,S)}.
	\]
	Further,
	\[
	\int_0^1|q'(t)|\,dt
	=
	\int_0^S|f'(s)g(s)+f(s)g'(s)|\,ds
	\le
	\|f'\|_{L^2}\|g\|_{L^2}+\|f\|_{L^2}\|g'\|_{L^2}.
	\]
	Using the low-energy assumption and
	Lemma~\ref{lem:continuous_embed_disc},
	\[
	\|f'\|_{L^2}\le C_*^{1/2}n^{-1}\|\boldsymbol{z}\|_2,
	\qquad
	\|g'\|_{L^2}\le C_*^{1/2}n^{-1}\|\boldsymbol{y}\|_2,
	\]
	and
	\[
	\|f\|_{L^2}\le C_\kappa \bar r^{1/2}\|\boldsymbol{z}\|_2,
	\qquad
	\|g\|_{L^2}\le C_\kappa \bar r^{1/2}\|\boldsymbol{y}\|_2.
	\]
	Hence
	\[
	\int_0^1|q'(t)|\,dt
	\le
	C_{\kappa,C_*}n^{-1}\|\boldsymbol{z}\|_2\|\boldsymbol{y}\|_2.
	\]
	Applying Lemma~\ref{lem:quadrature_discrepancy} to $q$ and multiplying
	by $S=n\bar r$ yields \eqref{eq:bilinear_mass_transfer}.
\end{proof}

\begin{lemma}[Continuous mean control]\label{lem:continuous_mean_control}
	If $\boldsymbol{z}\perp\boldsymbol{1}$ and $f=f_{\boldsymbol{z}}$, then
	\begin{equation}\label{eq:continuous_mean_control}
		\left|\frac1S\int_0^Sf(t)\,\mathrm dt\right|
		\le
		C_\kappa\left(\Delta+\frac1n\right)S^{1/2}\|f'\|_{L^2(0,S)}.
	\end{equation}
\end{lemma}

\begin{proof}
	Set $h(t):=f(St)$ on $[0,1]$. Since
	$\boldsymbol{z}\perp\boldsymbol{1}$,
	\[
	\frac1n\sum_{i=0}^{n-1}h(y_i)=\frac1n\sum_{i=0}^{n-1}z_i=0.
	\]
	Applying Lemma~\ref{lem:quadrature_discrepancy} to $h$ gives
	\[
	\left|\frac1S\int_0^Sf(t)\,dt\right|
	=
	\left|\int_0^1h(t)\,dt-\frac1n\sum_{i=0}^{n-1}h(y_i)\right|
	\le
	\left(\Delta+\frac1n\right)\int_0^1|h'(t)|\,dt.
	\]
	Since $h'(t)=Sf'(St)$,
	\[
	\int_0^1|h'(t)|\,dt
	=
	\int_0^S|f'(s)|\,ds
	\le
	S^{1/2}\|f'\|_{L^2(0,S)},
	\]
	and the claim follows.
\end{proof}

\subsection{Sinusoidal approximation of the first two modes}

\begin{theorem}[Fiedler vector is nearly sinusoidal in resistance arclength]\label{thm:fiedler-sine}
	There exists $\delta_0=\delta_0(\kappa)>0$ such that, if
	$\delta_n\le \delta_0$, then for every unit Fiedler vector
	$\boldsymbol{u}_1$ there exist a phase $\phi\in\mathbb{R}$ and a sign
	$\sigma\in\{\pm 1\}$ such that
	\begin{equation}\label{eq:fiedler_sine_sup}
		\max_{0\le i\le n-1}
		\left|
		u_{1,i}-\sigma\sqrt{\frac2n}
		\sin\left(\frac{2\pi s_i}{S}+\phi\right)
		\right|
		\le
		C_\kappa\sqrt{\frac{\delta_n}{n}}.
	\end{equation}
	Moreover,
	\begin{equation}\label{eq:lambda1_upper_for_later}
		\lambda_1\le \bar r\mu_1(1+C_\kappa\delta_n).
	\end{equation}
\end{theorem}

\begin{proof}
	Let $\theta_i:=2\pi s_i/S$ and
	$x_i:=\sqrt{2/n}\sin\theta_i$. By
	Lemma~\ref{lem:quadrature_discrepancy}, the centered vector
	$\widetilde{\boldsymbol{x}}:=P\boldsymbol{x}$ satisfies
	\[
	\|\widetilde{\boldsymbol{x}}\|_2^2=1+O_\kappa(\delta_n),
	\qquad
	\widetilde{\boldsymbol{x}}^\top L\widetilde{\boldsymbol{x}}
	=
	\bar r\mu_1(1+O_\kappa(\delta_n)).
	\]
	Rayleigh--Ritz therefore yields \eqref{eq:lambda1_upper_for_later}.
	
	Let $f:=f_{\boldsymbol{u}_1}$ and
	$g:=f/\|f\|_{L^2(0,S)}$. By
	Lemma~\ref{lem:low_energy_mass_transfer} and
	\eqref{eq:lambda1_upper_for_later},
	\[
	\|f\|_{L^2(0,S)}^2=\bar r(1+O_\kappa(\delta_n)),
	\qquad
	\int_0^S|g'(t)|^2\,dt\le \mu_1(1+C_\kappa\delta_n).
	\]
	Let
	\[
	m(g):=\frac1S\int_0^S g(t)\,dt,
	\qquad
	\widetilde g:=g-m(g),
	\qquad
	\widehat g:=\frac{\widetilde g}{\|\widetilde g\|_{L^2(0,S)}}.
	\]
	By Lemma~\ref{lem:continuous_mean_control},
	\[
	|m(g)|\le C_\kappa \delta_n S^{-1/2},
	\]
	so $\widehat g$ has mean zero, unit norm, and
	\[
	\int_0^S|\widehat g'(t)|^2\,dt
	\le \mu_1(1+C_\kappa\delta_n).
	\]
	Lemma~\ref{lem:circle_gap_disc} implies
	\[
	\operatorname{dist}_{L^2(0,S)}(\widehat g,\mathcal U)
	\le C_\kappa\sqrt{\delta_n}.
	\]
	Let
	\[
	P_{\mathcal U}\widehat g=a\psi_*,
	\qquad
	\|\psi_*\|_{L^2(0,S)}=1,
	\qquad
	h_\perp:=\widehat g-P_{\mathcal U}\widehat g.
	\]
	Then $h_\perp\perp\mathcal U$,
	$\|h_\perp\|_{L^2}\le C_\kappa\sqrt{\delta_n}$, and
	$a^2=1-\|h_\perp\|_{L^2}^2$, so $|a-1|\le C_\kappa\delta_n$.
	Since $\psi_*\in\mathcal U$, there exist $\phi\in\mathbb{R}$ and
	$\sigma\in\{\pm1\}$ such that
	\[
	\psi_*(t)=\sigma\sqrt{\frac2S}
	\sin\left(\frac{2\pi t}{S}+\phi\right).
	\]
	
	Moreover,
	\[
	\int_0^S|h_\perp'(t)|^2\,dt
	=
	\int_0^S|\widehat g'(t)|^2\,dt-a^2\mu_1
	\le
	C_\kappa\mu_1\delta_n.
	\]
	The periodic Sobolev inequality in one dimension gives
	\[
	\|h_\perp\|_{L^\infty(0,S)}
	\le
	C\left(
	S^{-1/2}\|h_\perp\|_{L^2(0,S)}
	+
	S^{1/2}\|h_\perp'\|_{L^2(0,S)}
	\right)
	\le
	C_\kappa\sqrt{\frac{\delta_n}{S}}.
	\]
	Since
	\[
	\widehat g-\psi_*=h_\perp+(a-1)\psi_*,
	\]
	and $\|\psi_*\|_{L^\infty}\le CS^{-1/2}$, we obtain
	\[
	\|\widehat g-\psi_*\|_{L^\infty(0,S)}
	\le
	C_\kappa\sqrt{\frac{\delta_n}{S}}.
	\]
	Returning from $\widehat g$ to $g$ introduces only an
	$O_\kappa(\delta_nS^{-1/2})$ error. Finally,
	$\|f\|_{L^2(0,S)}=\sqrt{\bar r}(1+O_\kappa(\delta_n))$ and
	$S=n\bar r$, which yield \eqref{eq:fiedler_sine_sup}.
\end{proof}

\begin{corollary}[Second mode is nearly a cosine]\label{cor:second-cosine}
	Under the assumptions of Theorem~\ref{thm:fiedler-sine}, fix the unit
	Fiedler vector $\boldsymbol{u}_1$ and the phase $\phi$ from
	\eqref{eq:fiedler_sine_sup}. There exists a unit eigenvector
	$\boldsymbol{u}_2$ associated with $\lambda_2$, chosen orthogonal to
	$\boldsymbol{u}_1$, and a sign $\tau\in\{\pm1\}$ such that
	\begin{equation}\label{eq:second_cosine_sup}
		\max_{0\le i\le n-1}
		\left|
		u_{2,i}-\tau\sqrt{\frac2n}
		\cos\left(\frac{2\pi s_i}{S}+\phi\right)
		\right|
		\le
		C_\kappa\sqrt{\frac{\delta_n}{n}}.
	\end{equation}
	Moreover,
	\begin{equation}\label{eq:lambda2_upper_for_later}
		\lambda_2\le \bar r\mu_1(1+C_\kappa\delta_n).
	\end{equation}
\end{corollary}

\begin{proof}
	For each phase $\psi\in\mathbb{R}$, define
	\[
	x_i^{(\psi)}:=
	\sqrt{\frac2n}\sin\left(\frac{2\pi s_i}{S}+\psi\right),
	\qquad
	\widetilde{\boldsymbol{x}}^{(\psi)}:=P\boldsymbol{x}^{(\psi)}.
	\]
	The quadrature estimate used in the proof of
	Theorem~\ref{thm:fiedler-sine} is uniform in $\psi$, and yields
	\[
	\|\widetilde{\boldsymbol{x}}^{(\psi)}\|_2^2=1+O_\kappa(\delta_n),
	\qquad
	\frac{(\widetilde{\boldsymbol{x}}^{(\psi)})^\top
		L\widetilde{\boldsymbol{x}}^{(\psi)}}
	{\|\widetilde{\boldsymbol{x}}^{(\psi)}\|_2^2}
	\le
	\bar r\mu_1(1+C_\kappa\delta_n).
	\]
	Hence the two-dimensional space
	\[
	W:=
	\operatorname{span}\left\{
	P\left(\sqrt{\frac2n}\sin\frac{2\pi s_i}{S}\right)_{i=0}^{n-1},
	P\left(\sqrt{\frac2n}\cos\frac{2\pi s_i}{S}\right)_{i=0}^{n-1}
	\right\}
	\subset \boldsymbol{1}^\perp
	\]
	has maximal Rayleigh quotient at most
	$\bar r\mu_1(1+C_\kappa\delta_n)$. The min--max principle gives
	\eqref{eq:lambda2_upper_for_later}.
	
	Let $\boldsymbol{u}_2$ be a unit eigenvector associated with $\lambda_2$
	and orthogonal to $\boldsymbol{u}_1$. Set $f_j:=f_{\boldsymbol{u}_j}$
	for $j=1,2$. By Lemma~\ref{lem:low_energy_mass_transfer},
	\eqref{eq:lambda1_upper_for_later}, and \eqref{eq:lambda2_upper_for_later},
	\[
	\|f_j\|_{L^2(0,S)}^2=\bar r(1+O_\kappa(\delta_n)),
	\qquad j=1,2.
	\]
	For $j=1,2$, write
	\[
	g_j:=\frac{f_j}{\|f_j\|_{L^2(0,S)}},
	\qquad
	m_j:=\frac1S\int_0^S g_j(t)\,dt,
	\qquad
	\widehat g_j:=\frac{g_j-m_j}{\|g_j-m_j\|_{L^2(0,S)}}.
	\]
	Lemma~\ref{lem:continuous_mean_control} gives
	$|m_j|\le C_\kappa\delta_n S^{-1/2}$, hence
	$\|g_j-m_j\|_{L^2}=1+O_\kappa(\delta_n)$. Also,
	\[
	\int_0^S|\widehat g_2'(t)|^2\,dt\le \mu_1(1+C_\kappa\delta_n).
	\]
	By Lemma~\ref{lem:circle_gap_disc},
	\begin{equation}\label{eq:g2_close_U}
		\operatorname{dist}_{L^2(0,S)}(\widehat g_2,\mathcal U)
		\le C_\kappa\sqrt{\delta_n}.
	\end{equation}
	The proof of Theorem~\ref{thm:fiedler-sine} gives
	\begin{equation}\label{eq:g1_close_sine_for_cosine}
		\left\|
		\widehat g_1-\psi_{\sin}
		\right\|_{L^2(0,S)}
		\le
		C_\kappa\sqrt{\delta_n},
	\end{equation}
	where
	\[
	\psi_{\sin}(t):=
	\sigma\sqrt{\frac2S}\sin\left(\frac{2\pi t}{S}+\phi\right),
	\qquad
	\psi_{\cos}(t):=
	\sqrt{\frac2S}\cos\left(\frac{2\pi t}{S}+\phi\right).
	\]
	Then $\{\psi_{\sin},\psi_{\cos}\}$ is an orthonormal basis of
	$\mathcal U$.
	
	Since $\boldsymbol{u}_1\perp \boldsymbol{u}_2$, the bilinear transfer
	estimate \eqref{eq:bilinear_mass_transfer} implies
	\[
	\langle f_1,f_2\rangle_{L^2(0,S)}=O_\kappa(\delta_n\bar r),
	\]
	and therefore
	\[
	\langle g_1,g_2\rangle_{L^2(0,S)}=O_\kappa(\delta_n).
	\]
	Subtracting the small means and renormalizing changes the inner product by
	only $O_\kappa(\delta_n)$, so
	\[
	\langle \widehat g_1,\widehat g_2\rangle_{L^2(0,S)}
	=
	O_\kappa(\delta_n).
	\]
	Using \eqref{eq:g1_close_sine_for_cosine}, we obtain
	\begin{equation}\label{eq:g2_sine_projection_small}
		\left|
		\langle \widehat g_2,\psi_{\sin}\rangle_{L^2(0,S)}
		\right|
		\le
		C_\kappa\sqrt{\delta_n}.
	\end{equation}
	
	Decompose
	\[
	\widehat g_2=a\psi_{\sin}+b\psi_{\cos}+h_\perp,
	\qquad h_\perp\perp \mathcal U.
	\]
	By \eqref{eq:g2_sine_projection_small}, $|a|\le C_\kappa\sqrt{\delta_n}$.
	By \eqref{eq:g2_close_U}, $\|h_\perp\|_{L^2}\le C_\kappa\sqrt{\delta_n}$.
	Since $\|\widehat g_2\|_{L^2}=1$,
	\[
	b^2=1-a^2-\|h_\perp\|_{L^2}^2=1+O_\kappa(\delta_n).
	\]
	Choose $\tau\in\{\pm1\}$ so that $\tau b\ge 0$. Then
	\begin{equation}\label{eq:g2_L2_close_cosine}
		\|\widehat g_2-\tau\psi_{\cos}\|_{L^2(0,S)}
		\le C_\kappa\sqrt{\delta_n}.
	\end{equation}
	
	Also,
	\[
	\int_0^S|h_\perp'(t)|^2\,dt
	\le C_\kappa\mu_1\delta_n.
	\]
	Indeed,
	\[
	\int_0^S|h_\perp'|^2
	=
	\int_0^S|\widehat g_2'|^2-\mu_1(a^2+b^2)
	\le
	\mu_1(1+C_\kappa\delta_n)-\mu_1(1-\|h_\perp\|_2^2)
	\le
	C_\kappa\mu_1\delta_n.
	\]
	Now
	\[
	\widehat g_2-\tau\psi_{\cos}
	=
	a\psi_{\sin}+(b-\tau)\psi_{\cos}+h_\perp.
	\]
	Using $|a|\le C_\kappa\sqrt{\delta_n}$, $|b-\tau|\le C_\kappa\delta_n$,
	and the bound on $h_\perp'$, we obtain
	\[
	\|(\widehat g_2-\tau\psi_{\cos})'\|_{L^2(0,S)}
	\le
	C_\kappa\sqrt{\mu_1\delta_n}.
	\]
	Together with \eqref{eq:g2_L2_close_cosine}, the periodic Sobolev
	inequality yields
	\[
	\|\widehat g_2-\tau\psi_{\cos}\|_{L^\infty(0,S)}
	\le
	C_\kappa\sqrt{\frac{\delta_n}{S}}.
	\]
	Returning from $\widehat g_2$ to $g_2$ adds only
	$O_\kappa(\delta_n S^{-1/2})$. Multiplying by
	$\|f_2\|_{L^2(0,S)}=\sqrt{\bar r}(1+O_\kappa(\delta_n))$ and evaluating
	at $t=s_i$ proves \eqref{eq:second_cosine_sup}.
\end{proof}

\begin{lemma}[Spectral localization under discrepancy control]\label{lem:spectral-localization}
	Under the assumptions of Theorem~\ref{thm:fiedler-sine},
	\begin{equation}\label{eq:lambda12_localization}
		\bar r\mu_1(1-C_\kappa\delta_n)
		\le \lambda_1,\lambda_2
		\le \bar r\mu_1(1+C_\kappa\delta_n).
	\end{equation}
	Consequently,
	\begin{equation}\label{eq:gamma_small_delta}
		0\le \lambda_2-\lambda_1\le C_\kappa\frac{\delta_n}{n^2}.
	\end{equation}
	Moreover, after shrinking $\delta_0$ if necessary, there exists
	$\rho_0=\rho_0(\kappa)\in(0,1)$ such that
	\begin{equation}\label{eq:lambda2_lambda3_gap}
		\lambda_2\le \rho_0\lambda_3,
	\end{equation}
	and there exists $L_\kappa>0$ such that
	\begin{equation}\label{eq:lambda2_Lkappa}
		\lambda_2\le \frac{L_\kappa}{n^2}.
	\end{equation}
\end{lemma}

\begin{proof}
	The upper bounds in \eqref{eq:lambda12_localization} are precisely
	\eqref{eq:lambda1_upper_for_later} and
	\eqref{eq:lambda2_upper_for_later}. We prove the lower bounds.
	
	Let $k\in\{1,2\}$, let $\boldsymbol{u}_k$ be a corresponding unit
	eigenvector, and set
	\[
	f_k:=f_{\boldsymbol{u}_k},
	\qquad
	g_k:=\frac{f_k}{\|f_k\|_{L^2(0,S)}}.
	\]
	The upper bounds already proved imply $\lambda_k\le C_\kappa n^{-2}$.
	Hence Lemma~\ref{lem:low_energy_mass_transfer} gives
	\begin{equation}\label{eq:mass_fk_lower_proof}
		\|f_k\|_{L^2(0,S)}^2=\bar r(1+O_\kappa(\delta_n)).
	\end{equation}
	Since $\boldsymbol{u}_k\perp\boldsymbol{1}$,
	Lemma~\ref{lem:continuous_mean_control} yields
	\[
	|m(g_k)|:=\left|\frac1S\int_0^S g_k(t)\,dt\right|
	\le C_\kappa\delta_n S^{-1/2}.
	\]
	Thus $\widetilde g_k:=g_k-m(g_k)$ has mean zero and
	\[
	\|\widetilde g_k\|_{L^2(0,S)}^2=1-S|m(g_k)|^2=1+O_\kappa(\delta_n).
	\]
	The circle Poincar\'e inequality, i.e., the first term in
	Lemma~\ref{lem:circle_gap_disc}, gives
	\[
	\int_0^S|g_k'(t)|^2\,dt
	=
	\int_0^S|\widetilde g_k'(t)|^2\,dt
	\ge
	\mu_1\|\widetilde g_k\|_{L^2(0,S)}^2
	\ge
	\mu_1(1-C_\kappa\delta_n).
	\]
	Combining this with \eqref{eq:mass_fk_lower_proof} and
	\eqref{eq:energy_identity} proves the lower bounds in
	\eqref{eq:lambda12_localization}. Subtracting the lower bound for
	$\lambda_1$ from the upper bound for $\lambda_2$, and using
	$S\asymp_\kappa n$, gives \eqref{eq:gamma_small_delta}.
	
	To prove the gap to $\lambda_3$, let
	\[
	\boldsymbol{x}\perp \operatorname{span}\{\boldsymbol{1},\boldsymbol{u}_1,\boldsymbol{u}_2\},
	\qquad
	\|\boldsymbol{x}\|_2=1.
	\]
	If $\boldsymbol{x}^\top L\boldsymbol{x}>8\bar r\mu_1$, then for
	sufficiently small $\delta_0$ this is already at least
	$4\bar r\mu_1(1-C_\kappa\delta_n)$. So assume instead that
	\[
	\boldsymbol{x}^\top L\boldsymbol{x}\le 8\bar r\mu_1.
	\]
	Then Lemma~\ref{lem:low_energy_mass_transfer} applies to
	$\boldsymbol{x}$. Let
	\[
	f:=f_{\boldsymbol{x}},
	\qquad
	g:=\frac{f}{\|f\|_{L^2(0,S)}},
	\qquad
	m(g):=\frac1S\int_0^S g(t)\,dt,
	\qquad
	\widehat g:=\frac{g-m(g)}{\|g-m(g)\|_{L^2(0,S)}}.
	\]
	As above,
	\[
	\|f\|_{L^2(0,S)}^2=\bar r(1+O_\kappa(\delta_n)),
	\qquad
	|m(g)|\le C_\kappa\delta_n S^{-1/2},
	\qquad
	\|g-m(g)\|_{L^2}=1+O_\kappa(\delta_n).
	\]
	Let $\widehat g_1,\widehat g_2$ be the centered normalized interpolants
	used in the proof of Corollary~\ref{cor:second-cosine}. Because
	$\boldsymbol{x}\perp \boldsymbol{u}_j$ for $j=1,2$, the bilinear
	estimate \eqref{eq:bilinear_mass_transfer} yields
	\[
	\langle g,\widehat g_j\rangle_{L^2(0,S)}=O_\kappa(\delta_n),
	\qquad j=1,2,
	\]
	and hence
	\begin{equation}\label{eq:x_hat_orth_ghat_j}
		\langle \widehat g,\widehat g_j\rangle_{L^2(0,S)}
		=
		O_\kappa(\delta_n),
		\qquad j=1,2.
	\end{equation}
	From Theorem~\ref{thm:fiedler-sine} and
	Corollary~\ref{cor:second-cosine},
	\[
	\|\widehat g_1-\psi_{\sin}\|_{L^2}\le C_\kappa\sqrt{\delta_n},
	\qquad
	\|\widehat g_2-\tau\psi_{\cos}\|_{L^2}\le C_\kappa\sqrt{\delta_n}.
	\]
	Combining these with \eqref{eq:x_hat_orth_ghat_j} gives
	\[
	|\langle \widehat g,\psi_{\sin}\rangle|
	+
	|\langle \widehat g,\psi_{\cos}\rangle|
	\le
	C_\kappa\sqrt{\delta_n}.
	\]
	Thus, writing
	\[
	\widehat g=a\psi_{\sin}+b\psi_{\cos}+h_\perp,
	\qquad h_\perp\perp\mathcal U,
	\]
	we have $a^2+b^2\le C_\kappa\delta_n$ and therefore
	\[
	\|h_\perp\|_{L^2(0,S)}^2\ge 1-C_\kappa\delta_n.
	\]
	Applying Lemma~\ref{lem:circle_gap_disc} to $\widehat g$,
	\[
	\int_0^S|\widehat g'(t)|^2\,dt
	\ge
	\mu_1+3\mu_1\|h_\perp\|_{L^2}^2
	\ge
	4\mu_1(1-C_\kappa\delta_n).
	\]
	Undoing the centering and normalization yields
	\[
	\boldsymbol{x}^\top L\boldsymbol{x}
	=
	\int_0^S|f'(t)|^2\,dt
	\ge
	4\bar r\mu_1(1-C_\kappa\delta_n).
	\]
	The min--max principle gives
	\[
	\lambda_3\ge 4\bar r\mu_1(1-C_\kappa\delta_n).
	\]
	Together with $\lambda_2\le \bar r\mu_1(1+C_\kappa\delta_n)$, this proves
	\eqref{eq:lambda2_lambda3_gap} after shrinking $\delta_0$. Finally,
	\eqref{eq:lambda2_Lkappa} follows from $S\asymp_\kappa n$ and
	\eqref{eq:lambda12_localization}.
\end{proof}

\subsection{Antipodal selection suppresses the second mode}

\begin{corollary}[Existence of a near-antipodal pair]\label{cor:near_antipodal_det}
	For every vertex $p$, there exists a vertex $q$ such that
	\[
	\left|d_R(p,q)-\frac S2\right|\le r_{\max}.
	\]
	Hence
	\[
	\mathcal A_\zeta:=
	\left\{
	(p,q):\left|d_R(p,q)-\frac S2\right|\le \zeta
	\right\}
	\]
	is nonempty whenever $\zeta\ge r_{\max}$.
\end{corollary}

\begin{proof}
	Choose $q$ to be the first mesh point encountered after $s_p+S/2$ in
	the positive cyclic direction. The overshoot is at most one edge
	resistance, hence at most $r_{\max}$.
\end{proof}

\begin{lemma}[Near-antipodal pairs carry a large Fiedler jump]\label{lem:beta1_lower_disc}
	Assume $\delta_n\le \delta_0$. Let
	\[
	r_{\max}\le \zeta\le S/8,
	\qquad
	\mathcal A_\zeta:=
	\left\{
	(p,q):\left|d_R(p,q)-\frac S2\right|\le \zeta
	\right\},
	\]
	and let $(\widehat p,\widehat q)\in \mathcal A_\zeta$ maximize
	$|\beta_1(p,q)|$ over $\mathcal A_\zeta$. Then, after shrinking
	$\delta_0$ if necessary,
	\begin{equation}\label{eq:beta1_lower_disc}
		\beta_1(\widehat p,\widehat q)^2\ge \frac{c_\kappa}{n}.
	\end{equation}
	More precisely, there exists a comparison pair $(p_*,q_*)\in\mathcal A_\zeta$
	such that
	\begin{equation}\label{eq:beta1_comparison_pair}
		|\beta_1(p_*,q_*)|
		\ge
		2\sqrt{\frac2n}
		\left(
		1-C_\kappa\left[
		\left(\frac{\zeta}{S}\right)^2+\sqrt{\delta_n}
		\right]
		\right).
	\end{equation}
\end{lemma}

\begin{proof}
	By Theorem~\ref{thm:fiedler-sine},
	\[
	u_{1,i}=
	\sigma\sqrt{\frac2n}\sin\theta_i+e_i^{(1)},
	\qquad
	\theta_i:=\frac{2\pi s_i}{S}+\phi,
	\qquad
	|e_i^{(1)}|\le C_\kappa\sqrt{\frac{\delta_n}{n}}.
	\]
	Choose $p_*$ so that $\theta_{p_*}$ is within $C_\kappa\eta$ of
	$-\pi/2$ modulo $2\pi$; this is possible because consecutive phases
	differ by at most $2\pi r_i/S\le C_\kappa\eta$. Let $q_*$ be the first
	vertex after resistance distance $S/2$ from $p_*$ in the positive
	direction. Since $\zeta\ge r_{\max}$, we have $(p_*,q_*)\in\mathcal A_\zeta$.
	
	Write
	\[
	\theta_{q_*}-\theta_{p_*}=\pi+\varepsilon_*,
	\qquad
	|\varepsilon_*|\le 2\pi\zeta/S.
	\]
	Because $\zeta\le S/8$, one has $|\varepsilon_*|\le \pi/4$. Hence
	\[
	\left|
	\sin\left(\frac{\theta_{p_*}-\theta_{q_*}}2\right)
	\right|
	=
	\cos(\varepsilon_*/2)
	\ge \cos(\pi/8).
	\]
	Also,
	\[
	\frac{\theta_{p_*}+\theta_{q_*}}2
	=
	\theta_{p_*}+\frac{\pi+\varepsilon_*}{2}
	\]
	is within $C_\kappa(\eta+\zeta/S)$ of an integer multiple of $\pi$.
	Using
	\[
	\sin\theta_{p_*}-\sin\theta_{q_*}
	=
	2\cos\left(\frac{\theta_{p_*}+\theta_{q_*}}2\right)
	\sin\left(\frac{\theta_{p_*}-\theta_{q_*}}2\right),
	\]
	we obtain
	\[
	|\beta_1(p_*,q_*)|\ge c_\kappa n^{-1/2},
	\]
	which proves \eqref{eq:beta1_lower_disc} for the maximizing pair.
	
	The refined estimate \eqref{eq:beta1_comparison_pair} follows from the
	Taylor expansions
	\[
	\cos(\varepsilon_*/2)=1+O\left((\zeta/S)^2\right),
	\qquad
	\cos\left(\frac{\theta_{p_*}+\theta_{q_*}}2\right)
	=
	\pm\left(1+O_\kappa((\zeta/S)^2+\eta^2)\right),
	\]
	together with the approximation error
	$O_\kappa(\sqrt{\delta_n/n})$.
\end{proof}

\begin{lemma}[Antipodal maximization suppresses the second mode]\label{lem:beta_ratio_disc}
	Under the assumptions of Lemma~\ref{lem:beta1_lower_disc},
	\begin{equation}\label{eq:beta_ratio_disc}
		\frac{\beta_2(\widehat p,\widehat q)^2}
		{\beta_1(\widehat p,\widehat q)^2}
		\le
		C_\kappa\left[
		\left(\frac{\zeta}{S}\right)^2+\sqrt{\delta_n}
		\right].
	\end{equation}
\end{lemma}

\begin{proof}
	Using Theorem~\ref{thm:fiedler-sine} and
	Corollary~\ref{cor:second-cosine},
	\[
	u_{1,i}=
	\sigma\sqrt{\frac2n}\sin\theta_i+e_i^{(1)},
	\qquad
	u_{2,i}=
	\tau\sqrt{\frac2n}\cos\theta_i+e_i^{(2)},
	\]
	with
	\[
	|e_i^{(1)}|+|e_i^{(2)}|
	\le
	C_\kappa\sqrt{\frac{\delta_n}{n}}.
	\]
	For any $(p,q)\in\mathcal A_\zeta$, orient the pair so that
	\[
	\theta_p-\theta_q=\pi+\varepsilon,
	\qquad
	|\varepsilon|\le C\zeta/S,
	\]
	and define
	\[
	\alpha:=\frac{\theta_p+\theta_q}{2}.
	\]
	Then
	\begin{align*}
		\beta_1(p,q)
		&=
		2\sigma\sqrt{\frac2n}\cos\alpha
		+
		O_\kappa\!\left(
		\sqrt{\frac1n}\left[
		\left(\frac{\zeta}{S}\right)^2+\sqrt{\delta_n}
		\right]\right),\\
		\beta_2(p,q)
		&=
		-2\tau\sqrt{\frac2n}\sin\alpha
		+
		O_\kappa\!\left(
		\sqrt{\frac1n}\left[
		\left(\frac{\zeta}{S}\right)^2+\sqrt{\delta_n}
		\right]\right).
	\end{align*}
	Applying this to $(\widehat p,\widehat q)$, the maximizing property and
	Lemma~\ref{lem:beta1_lower_disc} imply that
	$|\cos\widehat\alpha|$ is within
	$C_\kappa[(\zeta/S)^2+\sqrt{\delta_n}]$ of $1$. Therefore
	\[
	|\sin\widehat\alpha|^2
	\le
	C_\kappa\left[
	\left(\frac{\zeta}{S}\right)^2+\sqrt{\delta_n}
	\right].
	\]
	Using the expansion for $\beta_2$ and the lower bound
	\eqref{eq:beta1_lower_disc} proves \eqref{eq:beta_ratio_disc}.
\end{proof}

\begin{lemma}[Verification package for Theorem~\ref{thm:comparison_two_mode}]\label{lem:verify_abstract_chord}
	Assume \eqref{eq:det_kappa_bounds}, \eqref{eq:w_native_scale}, and
	$\delta_n\le \delta_0$. Let
	\[
	r_{\max}\le \zeta\le S/8,
	\qquad
	(\widehat p,\widehat q)\in
	\arg\max_{(p,q)\in\mathcal A_\zeta}|\beta_1(p,q)|.
	\]
	Define
	\begin{equation}\label{eq:eps_zn}
		\varepsilon_{\zeta,n}:=
		C_\kappa\left[
		\left(\frac{\zeta}{S}\right)^2+\sqrt{\delta_n}
		\right].
	\end{equation}
	After shrinking $\delta_0=\delta_0(\kappa,C_w)$ and increasing
	$n_0=n_0(\kappa,C_w)$ if necessary, all hypotheses of
	Theorem~\ref{thm:comparison_two_mode} hold for
	$(\widehat p,\widehat q)$ with $\varepsilon=\varepsilon_{\zeta,n}$.
	In particular,
	\begin{align}
		&\lambda_2\le \rho_0\lambda_3,
		\qquad
		\lambda_2\le L_\kappa n^{-2},
		\qquad
		\beta_1(\widehat p,\widehat q)^2\ge c_\kappa n^{-1},
		\label{eq:verify_first}\\
		&T_{3+}(\widehat p,\widehat q)\le A_\kappa n,
		\qquad
		w\ge w_0:=1/C_w,
		\qquad
		\beta_2(\widehat p,\widehat q)^2
		\le
		\varepsilon_{\zeta,n}\beta_1(\widehat p,\widehat q)^2,
		\label{eq:verify_second}\\
		&(\lambda_2-\lambda_1)\left(
		\frac1w+\frac{T_{3+}(\widehat p,\widehat q)}{1-\rho_0}
		\right)
		\le \theta_0\beta_1(\widehat p,\widehat q)^2
		\label{eq:verify_dominance}
	\end{align}
	for some fixed $\theta_0\in(0,1)$.
\end{lemma}

\begin{proof}
	The spectral bounds in \eqref{eq:verify_first} follow from
	Lemma~\ref{lem:spectral-localization}, while the lower bound on
	$\beta_1^2$ follows from Lemma~\ref{lem:beta1_lower_disc}. The smallness
	of $\beta_2^2/\beta_1^2$ is exactly
	Lemma~\ref{lem:beta_ratio_disc}. The weight condition follows from
	\eqref{eq:w_native_scale} because $\bar r\le 1$, hence
	$w\ge 1/C_w$.
	
	It remains to bound $T_{3+}$ and verify the dominance condition. Since
	each term in the modal expansion of $R_{pq}$ is nonnegative,
	\[
	T_{3+}(\widehat p,\widehat q)\le R_{\widehat p\widehat q}.
	\]
	For a cycle, Lemma~\ref{lem:Rab_cycle} gives
	\[
	R_{pq}=\frac{A_{pq}(S-A_{pq})}{S}\le \frac{S}{4}.
	\]
	Because $S\le n$, this implies $T_{3+}\le n/4$, so one may take
	$A_\kappa=1$ after enlarging constants.
	
	Finally, by \eqref{eq:gamma_small_delta},
	\[
	\gamma:=\lambda_2-\lambda_1\le C_\kappa\frac{\delta_n}{n^2}.
	\]
	Therefore
	\[
	\gamma\left(
	\frac1w+\frac{T_{3+}}{1-\rho_0}
	\right)
	\le
	C_{\kappa,C_w}\left(
	\frac{\delta_n}{n^2}+\frac{\delta_n}{n}
	\right)
	\le
	C_{\kappa,C_w}\frac{\delta_n}{n}.
	\]
	By \eqref{eq:beta1_lower_disc},
	\[
	\beta_1(\widehat p,\widehat q)^2\ge \frac{c_\kappa}{n}.
	\]
	Shrinking $\delta_0$ so that
	$C_{\kappa,C_w}\delta_0\le \theta_0 c_\kappa$ proves
	\eqref{eq:verify_dominance}.
\end{proof}

\begin{theorem}[High-probability discrepancy bound]\label{thm:discrepancy_iid_kappa}
	Under Eq.~\eqref{eq:iid_kappa_assumption}, there exists $C_\kappa>0$
	such that, for every $x\ge 0$,
	\begin{equation}\label{eq:Delta_prob_kappa}
		\Prob\left(
		\Delta>C_\kappa\sqrt{\frac{x+\log n}{n}}
		\right)\le 4e^{-x}.
	\end{equation}
	Moreover, $\eta\le \kappa/n$ almost surely, and hence, after enlarging
	$C_\kappa$,
	\begin{equation}\label{eq:delta_prob_kappa}
		\Prob\left(
		\delta_n>C_\kappa\sqrt{\frac{x+\log n}{n}}
		\right)\le 4e^{-x}.
	\end{equation}
\end{theorem}

\begin{proof}
	Let $\mu:=\E r_0$. For each $(p,\ell)$, Hoeffding's inequality gives
	\[
	\Prob\!\left(|A_{p,\ell}-\ell\mu|>t\right)
	\le
	2\exp\!\left(
	-\frac{2t^2}{\ell(1-\kappa^{-1})^2}
	\right).
	\]
	Since $\ell\le n$, taking
	$t=C_\kappa\sqrt{n(x+\log n)}$ and a union bound over the at most
	$n^2$ pairs $(p,\ell)$ yields $\max_{p,\ell}|A_{p,\ell}-\ell\mu|
	\le
	C_\kappa\sqrt{n(x+\log n)}$ with probability at least $1-2e^{-x}$. A second Hoeffding bound gives $|S-n\mu|\le C_\kappa\sqrt{n(x+\log n)}$ with the same probability. On the intersection,
	\begin{equation*}
		\left|A_{p,\ell}-\frac{\ell}{n}S\right|
		\le
		|A_{p,\ell}-\ell\mu|+\frac{\ell}{n}|S-n\mu|
		\le
		C_\kappa\sqrt{n(x+\log n)}
	\end{equation*}
	for all $(p,\ell)$, hence $D\le C_\kappa\sqrt{n(x+\log n)}$. Since $S\ge n/\kappa$ almost surely,
	Eq.~\eqref{eq:Delta_prob_kappa} follows. Also $r_{\max}\le 1$ and
	$S\ge n/\kappa$ almost surely, so $\eta=r_{\max}/S\le \kappa/n$, which
	proves Eq.~\eqref{eq:delta_prob_kappa}.
\end{proof}

\begin{corollary}[Uniform spacing of random resistance coordinates]\label{cor:grid_spacing_kappa}
	Under Eq.~\eqref{eq:iid_kappa_assumption}, for every $x\ge 0$,
	\[
	\Prob\left(
	\max_{0\le i\le n}
	\left|
	\frac{s_i}{S}-\frac{i}{n}
	\right|
	>
	C_\kappa\sqrt{\frac{x+\log n}{n}}
	\right)\le 4e^{-x}.
	\]
\end{corollary}

\begin{proof}
	This is immediate from Eq.~\eqref{eq:node_discrepancy} and
	Theorem~\ref{thm:discrepancy_iid_kappa}.
\end{proof}

\begin{theorem}[High-probability entry into the deterministic regime]\label{thm:random_to_det_transfer}
	Assume Eq.~\eqref{eq:iid_kappa_assumption}. Then, for every $x\ge 0$,
	with probability at least $1-4e^{-x}$,
	\begin{equation}\label{eq:delta_random_event}
		\delta_n\le C_\kappa\sqrt{\frac{x+\log n}{n}}.
	\end{equation}
	If the right-hand side is at most the deterministic threshold
	$\delta_0$, then all conclusions of
	Section~\ref{subsec:deterministic_discrepancy} hold on this event with
	this upper bound substituted for $\delta_n$.
\end{theorem}
\begin{proof}
	Equation~\eqref{eq:delta_random_event} is exactly
	Eq.~\eqref{eq:delta_prob_kappa}, and the rest is the deterministic theory
	applied on that event.
\end{proof}

\paragraph{Interpretation.}
The high-probability result concerns the resistance-discrepancy parameter
$\delta_n$, not the pointwise variability of individual conductances. Thus,
the i.i.d. bounded model does not require the edge conductances to be nearly
equal. Rather, for each fixed conductance ratio $\kappa$, the normalized
resistance arclength coordinates become uniformly distributed around the
cycle with high probability as $n$ grows. In particular,
Theorem~\ref{thm:random_to_det_transfer} shows that
\[
\delta_n = O_{\mathbb P}\!\left(\sqrt{\frac{\log n}{n}}\right),
\]
and therefore the deterministic small-discrepancy regime is reached with
probability tending to one whenever the deterministic threshold
$\delta_0$ is fixed. The constants depend on $\kappa$, so the statement
should be interpreted for fixed or moderately bounded conductance ratios,
not for regimes in which $\kappa$ grows arbitrarily with $n$.

\section{Proofs of main-text statements}
\label{app:main-proofs}
For readability, the main text states the principal claims and places the proofs here.

\subsection{Proof of Lemma~\ref{lem:Rab_cycle}}
\begin{proof}[Proof of Lemma~\ref{lem:Rab_cycle}]
	The terminals $v_a$ and $v_b$ split the cycle into two internally disjoint paths whose total resistances are $d(a,b)$ and $d(b,a)$.
	The two paths are in parallel, so
	\[
	R_{ab}=\left(\frac1{d(a,b)}+\frac1{d(b,a)}\right)^{-1}
	=\frac{d(a,b)d(b,a)}{T}.
	\]
\end{proof}

\subsection{Proof of Lemma~\ref{lem:res_kir_update}}
\begin{proof}[Proof of Lemma~\ref{lem:res_kir_update}]
	Because $\boldsymbol{b}^\top\boldsymbol{1}=0$, the pseudoinverse Sherman--Morrison identity gives $L_{p,q}(w)^\dagger
	=L^\dagger-
	\frac{wL^\dagger\boldsymbol{b}\boldsymbol{b}^\top L^\dagger}{1+w\boldsymbol{b}^\top L^\dagger\boldsymbol{b}}$. Since $\boldsymbol{b}^\top L^\dagger\boldsymbol{b}=R_{pq}$, part (i) follows.
	For $\boldsymbol{g}_{uv}:=\boldsymbol{e}_u-\boldsymbol{e}_v$, $R_{uv}(w)=R_{uv}-\frac{w(\boldsymbol{g}_{uv}^\top L^\dagger\boldsymbol{b})^2}{1+wR_{pq}}$.
	The identity $(\boldsymbol{e}_u-\boldsymbol{e}_v)^\top L^\dagger(\boldsymbol{e}_p-\boldsymbol{e}_q)
	=\frac12(R_{uq}+R_{vp}-R_{up}-R_{vq})$ gives part (ii).
	Summing part (iii) over all unordered pairs gives Eq.\eqref{eq:Kf_update_cycle_compact}--Eq.\eqref{eq:Kf_improvement_full}, and Eq.\eqref{eq:Kf_cycle_sum} follows from Lemma~\ref{lem:Rab_cycle}.
\end{proof}

\subsection{Proof of Theorem~\ref{thm:R_Kf_quant}}
\begin{proof}[Proof of Theorem~\ref{thm:R_Kf_quant}]
	Parts (i) and (ii) follow by differentiating Eq.\eqref{eq:Rpq_update_cycle}.
	For (iii), the coefficient $w/(1+wR_{pq})$ is increasing and concave in $w$.
	The square-sum in Eq.\eqref{eq:Kf_improvement_full} is strictly positive: taking the term $\{u,v\}=\{p,q\}$ gives
	\[
	R_{pq}+R_{qp}-R_{pp}-R_{qq}=2R_{pq}>0.
	\]
	Therefore $\mathcal I_{p,q}(w)$ is strictly increasing for $w>0$, and $K_f(L_{p,q}(w))=K_f(L)-\mathcal I_{p,q}(w)$ is strictly decreasing.
\end{proof}

\subsection{Proof of Corollary~\ref{cor:selection_RK}}
\begin{proof}[Proof of Corollary~\ref{cor:selection_RK}]
	Part (i) concerns the endpoint resistance reduction. From \eqref{eq:Rpq_update_cycle},
	\[
	\Delta R_{pq}(w)=R_{pq}-R_{pq}(w)=\frac{wR_{pq}^2}{1+wR_{pq}}.
	\]
	For a fixed pair $(p,q)$, this quantity is strictly increasing in $w$, so the optimal weight is $w=\hat{w}$. With $w=\hat{w}$ fixed and $x=R_{pq}>0$,
	\[
	\frac{\mathrm d}{\mathrm dx}\frac{\hat{w}x^2}{1+\hat{w}x}
	=\frac{\hat{w}x(2+\hat{w}x)}{(1+\hat{w}x)^2}>0.
	\]
	Hence the endpoint resistance reduction is maximized by maximizing the initial resistance $R_{pq}=A_{pq}B_{pq}/T$, which measures how balanced the two complementary arcs are.
	
	Part (ii) follows from the exact Kirchhoff index update \eqref{eq:Kf_update_cycle_compact}--\eqref{eq:Kf_improvement_full}. Theorem~\ref{thm:R_Kf_quant} establishes that the improvement $\mathcal{I}_{p,q}(w)$ is strictly increasing in $w$, so the optimal weight is again $w=\hat{w}$. The optimal pair $(p,q)$ is then the one that maximizes $\mathcal{I}_{p,q}(\hat{w})$ as given by \eqref{eq:Kf_improvement_full}. The presence of the denominator $1+\hat{w} R_{pq}$ implies that a simple maximization of the square-sum alone does not generally identify the optimal chord.
\end{proof}

\subsection{Proof of Lemma~\ref{lem:gain_ceiling}}
\begin{proof}[Proof of Lemma~\ref{lem:gain_ceiling}]
	The bound $0\le \Delta_{p,q}(w)\le \gamma$ is exactly the interlacing
	statement in Lemma~\ref{lem:lambda_update}. For the second estimate,
	let $\boldsymbol{u}_1$ be a unit Fiedler vector of $L$. Since
	$\boldsymbol{u}_1\perp \boldsymbol{1}$, the Rayleigh--Ritz
	characterization gives
	\[
	\lambda_1(L_{p,q}(w))
	\le
	\boldsymbol{u}_1^\top L_{p,q}(w)\boldsymbol{u}_1
	=
	\lambda_1(L)+w\boldsymbol{u}_1^\top
	\boldsymbol{b}\boldsymbol{b}^\top\boldsymbol{u}_1.
	\]
	Because $\boldsymbol{u}_1^\top\boldsymbol{b}=u_{1,p}-u_{1,q}$, the
	claim follows.
\end{proof}

\subsection{Proof of Theorem~\ref{thm:lambda_quant}}
\begin{proof}[Proof of Theorem~\ref{thm:lambda_quant}]
	Equation~\eqref{eq:var_lam_w} follows from Eq.~\eqref{eq:rr_lambda1} and
	\[
	\boldsymbol{x}^\top L_{p,q}(w)\boldsymbol{x}
	=
	\boldsymbol{x}^\top L\boldsymbol{x}+w(x_p-x_q)^2.
	\]
	Since $\lambda_1(w)$ is the pointwise infimum of affine functions of
	$w$, it is concave; monotonicity follows because the penalty term is
	nonnegative. The derivative formula is the Hellmann--Feynman identity.
	The limit is the standard penalty limit enforcing the constraint
	$x_p=x_q$.
\end{proof}

\subsection{Proof of Lemma~\ref{lem:two_mode_resistance}}
\begin{proof}[Proof of Lemma~\ref{lem:two_mode_resistance}]
	For $k\ge 3$ and $\mu<\lambda_2$,
	\[
	\lambda_k-\mu\ge \lambda_k-\lambda_2\ge (1-\rho_0)\lambda_k.
	\]
	Summing the resulting termwise bound proves the claim.
\end{proof}

\subsection{Proof of Theorem~\ref{thm:comparison_two_mode}}
\begin{proof}[Proof of Theorem~\ref{thm:comparison_two_mode}]
	Let $\mu:=\lambda_1(L_{p,q}(w))=\lambda_1+t$, $0\le t\le \gamma$,
	and define $d:=\lambda_2-\mu=\gamma-t$.
	Then $\gamma-\Delta_{p,q}(w)=d$. If $d=0$, there is nothing to prove,
	so assume $d>0$. Since $\lambda_1<\lambda_2$, the eigenvalue
	$\lambda_1$ is simple. Moreover, \eqref{eq:comparison_hyp1} gives
	$\beta_1^2>0$ and $w>0$, so the rank-one perturbation is not orthogonal
	to the Fiedler eigendirection. Hence the eigenvalue branch issuing from
	$\lambda_1$ is strictly shifted upward. Therefore the updated algebraic
	connectivity cannot remain at the pole $\lambda_1$, and we have
	$\mu\in(\lambda_1,\lambda_2)$. The secular equation
	\eqref{eq:secular_spectral} is therefore valid at $\mu$ and becomes
	\begin{equation}\label{eq:secular_d_form}
		1-\frac{w\beta_1^2}{\gamma-d}+\frac{w\beta_2^2}{d}+wH(\mu)=0,
	\end{equation}
	where
	\[
	H(\mu):=\sum_{k=3}^{n-1}\frac{\beta_k^2}{\lambda_k-\mu}\ge 0.
	\]
	By Lemma~\ref{lem:two_mode_resistance},
	\begin{equation}\label{eq:H_tail_bound}
		H(\mu)\le \frac{T_{3+}}{1-\rho_0}.
	\end{equation}
	Rearranging \eqref{eq:secular_d_form} gives
	\begin{equation}\label{eq:secular_rearranged}
		\frac{\beta_1^2}{\gamma-d}
		=
		\frac1w+\frac{\beta_2^2}{d}+H(\mu).
	\end{equation}
	Multiplying by $d(\gamma-d)$ yields
	\[
	d\beta_1^2
	=
	(\gamma-d)\beta_2^2
	+
	d(\gamma-d)\left(\frac1w+H(\mu)\right).
	\]
	Since $0\le \gamma-d\le \gamma$, combining this with
	\eqref{eq:H_tail_bound} and \eqref{eq:comparison_dominance} gives
	\[
	d\beta_1^2
	\le
	\gamma\beta_2^2
	+
	d\gamma\left(\frac1w+\frac{T_{3+}}{1-\rho_0}\right)
	\le
	\gamma\beta_2^2+\theta_0 d\beta_1^2.
	\]
	Hence
	\[
	d\le \frac{\gamma\beta_2^2}{(1-\theta_0)\beta_1^2}.
	\]
	Using \eqref{eq:comparison_small_beta2} proves
	\eqref{eq:ceiling_deficit_gamma}.
	
	Finally, \eqref{eq:comparison_hyp1} implies
	\[
	\gamma\le \lambda_2\le \frac{L_0}{n^2}
	\le \frac{L_0}{w_0}\frac{w}{n},
	\]
	so \eqref{eq:ceiling_deficit_gamma} yields
	\eqref{eq:ceiling_deficit_wn}.
\end{proof}

\subsection{Proof of Theorem~\ref{thm:main_optimal_disc}}
\begin{proof}[Proof of Theorem~\ref{thm:main_optimal_disc}]
	Let $\varepsilon_{\zeta,n}$ be defined by \eqref{eq:eps_zn}. By
	Lemma~\ref{lem:verify_abstract_chord}, all hypotheses of
	Theorem~\ref{thm:comparison_two_mode} hold for
	$(\widehat p,\widehat q)$ with
	$\varepsilon=\varepsilon_{\zeta,n}$. Hence
	\[
	(\lambda_2-\lambda_1)-\Delta_{\widehat p,\widehat q}(w)
	\le
	C(\lambda_2-\lambda_1)\varepsilon_{\zeta,n}.
	\]
	Corollary~\ref{cor:ceiling_to_opt} converts this ceiling-deficit estimate
	into \eqref{eq:main_optimal_disc_gamma}. The additive estimate
	\eqref{eq:main_optimal_disc_additive} follows from
	\eqref{eq:ceiling_deficit_wn}. For $\zeta=r_{\max}$,
	\[
	\left(\frac{\zeta}{S}\right)^2=\eta^2\le \eta\le \delta_n\le \sqrt{\delta_n}
	\]
	after taking $\delta_0\le 1$, which gives
	\eqref{eq:main_optimal_disc_additive_native}. Finally, admissibility of
	$\zeta=r_{\max}$ follows from $r_{\max}\le 1$ and $S\ge n/\kappa$.
\end{proof}

\subsection{Proof of Corollary~\ref{cor:main_optimal_disc_ratio}}
\begin{proof}[Proof of Corollary~\ref{cor:main_optimal_disc_ratio}]
	By Theorem~\ref{thm:main_optimal_disc},
	\[
	0\le
	\Delta^\star(w)-\Delta_{\widehat p,\widehat q}(w)
	\le
	C_{\kappa,C_w}(\lambda_2-\lambda_1)
	\left[
	\left(\frac{\zeta}{S}\right)^2+\sqrt{\delta_n}
	\right].
	\]
	Divide by $\Delta^\star(w)$ and use
	Eq.~\eqref{eq:ratio_nondegenerate}. The upper bound is immediate from the
	definition of $(p^\star,q^\star)$. For $\zeta=r_{\max}$,
	\[
	(r_{\max}/S)^2=\eta^2\le \eta\le \delta_n\le \sqrt{\delta_n},
	\]
	after shrinking $\delta_0$ so that $\delta_0\le 1$.
\end{proof}

\subsection{Proof of Theorem~\ref{thm:main_optimal_random}}
\begin{proof}[Proof of Theorem~\ref{thm:main_optimal_random}]
	By Theorem~\ref{thm:random_to_det_transfer}, with probability at least
	$1-4e^{-x}$, $\delta_n\le C_\kappa\sqrt{\frac{x+\log n}{n}}$.
	
	On this event, Theorem~\ref{thm:main_optimal_disc} yields
	\[
	0\le
	\max_{\{p,q\}\in\mathcal{E}_{\mathrm{ch}}}\Delta_{p,q}(w)-\Delta_{\widehat p,\widehat q}(w)
	\le
	C_{\kappa,C_w}\frac{w}{n}
	\left[
	\left(\frac{\zeta}{S}\right)^2+\sqrt{\delta_n}
	\right].
	\]
	Substituting the bound on $\delta_n$ gives
	Eq.~\eqref{eq:main_random_additive}. For $\zeta=r_{\max}$, $\left(\frac{\zeta}{S}\right)^2=\eta^2\le \eta\le \delta_n$, which is absorbed by $\sqrt{\delta_n}$ once $\delta_0\le 1$. This
	proves Eq.~\eqref{eq:main_random_additive_native}.
\end{proof}
\end{document}